\documentclass[11pt]{article}
\usepackage{amsthm,amsmath,amsfonts,amscd}
\usepackage[T1]{fontenc}
\usepackage{tikz,verbatim}
\usepackage{multicol,parskip}
\usepackage{subfigure,fancyhdr}
\usetikzlibrary{decorations.pathreplacing}
\usetikzlibrary{decorations.markings}
\usetikzlibrary{fadings}
\usetikzlibrary{patterns}
\usetikzlibrary{shapes}
\usetikzlibrary{positioning,arrows}
\usepackage{wasysym,graphicx}

\newcommand{\rmut}{\mathfrak{R}}

\newcommand{\QG}{$\QQ$-Gorenstein }
\newcommand{\CC}{\mathbb{C}}
\newcommand{\QQ}{\mathbb{Q}}
\newcommand{\RR}{\mathbb{R}}

\newcommand{\ZZ}{\mathbb{Z}}
\newcommand{\cp}[1]{\mathbb{CP}^{#1}}

\newcommand{\matr}[4]{\left(\begin{array}{cc}#1 & #2\\ #3 & #4\end{array}\right)}

\newcommand{\vect}[2]{\left(\begin{array}{c}#1\\#2\end{array}\right)}

\def\chk#1{#1^{\smash{\scalebox{.7}[1.4]{\rotatebox{90}{\textnormal{\guilsinglleft}}}}}}

\newtheorem{Theorem}{Theorem}[section]
\newtheorem{Lemma}[Theorem]{Lemma}
\newtheorem{Corollary}[Theorem]{Corollary}

\theoremstyle{remark}
\newtheorem{Remark}[Theorem]{Remark}
\theoremstyle{definition}
\newtheorem{Definition}[Theorem]{Definition}
\newtheorem{exm}[Theorem]{Example}

\title{Antiflips, mutations, and unbounded symplectic embeddings of rational homology balls}

\author{Jonathan David Evans and Giancarlo Urz\'{u}a}

\begin{document}
\maketitle
%% Abstract in English
\begin{abstract}
  The Milnor fibre of a \(\mathbb{Q}\)-Gorenstein smoothing of a Wahl
  singularity is a rational homology ball $B_{p,q}$. For a canonically
  polarised surface of general type $X$, it is known that there are
  bounds on the number $p$ for which $B_{p,q}$ admits a symplectic
  embedding into $X$. In this paper, we give a recipe to construct
  unbounded sequences of symplectically embedded $B_{p,q}$ into
  surfaces of general type equipped with {\em non-canonical}
  symplectic forms. Ultimately, these symplectic embeddings come from
  Mori's theory of flips, but we give an interpretation in terms of
  almost toric structures and mutations of polygons. The key point is
  that a flip of surfaces, as studied by Hacking, Tevelev and
  Urz\'{u}a \cite{HTU}, can be formulated as a combination of
  mutations of an almost toric structure and deformation of the
  symplectic form.
\end{abstract}

% Abstract in French
%\begin{altabstract}
%  La fibre de Milnor d'un lissage \(\mathbb{Q}\)-Gorenstein d'une
%  singularit\'{e} de Wahl est une boule d'homologie rationelle
%  \(B_{p,q}\). Si \(X\) est une surface de type g\'{e}n\'{e}ral
%  polaris\'{e}e canoniquement, l'ensemble des entiers \(p\) pour
%  lesquels il existe un plongement symplectique de \(B_{p,q}\) dans
%  \(X\) est born\'{e}. Dans cet article, nous montrons comment
%  construire une suite non-born\'{e}e de boules d'homologie
%  rationnelles plong\'{e}es symplectiquement dans des surfaces de type
%  g\'{e}n\'{e}ral munies de formes symplectiques non-canoniques. Ces
%  plongements proviennent de la th\'{e}orie de Mori sur les flips,
%  mais nous les interpr\'{e}tons en termes de structures presque
%  toriques et de mutations de polygones. Un flip de surfaces tel que
%  ceux \'{e}tudi\'{e}s par Hacking, Tevelev et Urz\'{u}a peut \^{e}tre
%  d\'{e}compos\'{e} en une succession de mutations de structure
%  presque torique et de deformations de la forme symplectique.
%\end{altabstract}

\section{Introduction}

\subsection{Setting and results}

Wahl singularities are the cyclic quotient surface singularities
admitting a \QG smoothing whose Milnor fibre is a rational homology
ball \cite{LooijengaWahl,Wahl}. The rational homology balls $B_{p,q}$
arising this way are Stein manifolds whose Lagrangian skeleton is a
certain cell complex called a {\em Lagrangian pinwheel} $L_{p,q}$,
with one 1-cell and one 2-cell
\cite{ESMarkov,KhoBounds,LekiliMaydanskiy}. If $X$ is an algebraic
surface, one can hope to understand which Wahl singularities can
appear in degenerations of $X$ by studying the symplectic embeddings
of rational homology balls $B_{p,q}$ (or, equivalently, Lagrangian
embeddings of pinwheels $L_{p,q}$) in $X$.

In \cite{ESbound}, it was proved that for a symplectic 4-manifold
$(X,\omega)$, with $b^+>1$ and $[\omega]=K_X$ (which one can think of
as a surface of general type with positive geometric genus), there is
a bound on the integers $p$ for which there is a symplectic embedding
of the rational homology ball $B_{p,q}$ into $X$ (equivalently, by
(Lemmas 3.3, 3.4, \cite{KhoBounds}), a Lagrangian pinwheel of type
$L_{p,q}$). Namely, if $\ell$ denotes the length of the continued
fraction expansion of $\frac{p^2}{pq-1}$, we have \[\ell\leq 4K^2+7.\]
This implies a bound on $p$. (Compare with the similar proof of the
better bound $\ell\leq 4K^2+1$ in the context of algebraic geometry in
\cite{RanaUrzua}.)

In the current paper, we will show that the hypothesis $[\omega]=K_X$
in this result is necessary. We do this by exhibiting symplectic
4-manifolds which admit sequences of embedded Lagrangian pinwheels
$\{L_{p_i,q_i}\}_{i=1}^\infty$ where $p_i\to\infty$.

The sequences $(p_i,q_i)$ in question all satisfy a certain recursion
relation which arises in Mori's theory of flips; we call them {\em
Mori sequences}. A Mori sequence is determined by its first two terms;
we therefore write $M(p_1,q_1;p_2,q_2)$ to specify a Mori
sequence. See Section \ref{sct:mori} for the definition.

Our construction applies very widely and yields unbounded Lagrangian
pinwheels in any surface of general type which arises as a smoothing
of a suitable KSBA-stable surface. The only requirement is that the
KSBA-stable surface has at worst Wahl singularities and contains a
suitable rational curve passing through at most two of these
singularities (see Theorem \ref{thm:ksba} for a precise statement). We
illustrate the applicability of the construction with two examples,
one with $b^+>1$ and one with $b^+=1$:

\begin{Theorem}\label{thm:mainthm}
In each of the cases listed below, $X$ carries a symplectic form
$\omega$ for which there is a sequence of Lagrangian pinwheels
$L_{p_i,q_i}\subset (X,\omega)$, for the given Mori sequence
$\{(p_i,q_i)\}_{i=1}^{\infty}$:
\begin{itemize}
\item $X$ is a quintic surface ($b^+=9$), with Mori sequence
\end{itemize}
\[M(1,0;5,3)=\{(1,1),(5,3),(14,9),(37,24),(97,63),(254,165),\ldots\}.\]
\begin{itemize}
\item $X$ is a simply-connected Godeaux surface\footnote{A {\em Godeaux
surface} is a minimal surface of general type with $K^2=1$; the
simply-connected ones are homeomorphic to $\cp{2}\#
8\overline{\cp{2}}$.} ($b^+=1$), with Mori sequence
\end{itemize}
\[M(5,2;39,17)=\{(5,2),(39,17),(268,49),(1837,326),(12591,2233),\ldots\}.\]

\end{Theorem}
\begin{Remark}
In fact, with essentially no extra work, we can also find a
symplectic form on the quintic containing the Mori sequence
\[M(2,1;7,5)=\{(2,1),(7,5),(19,14),(50,37),(131,97),(343,254),\ldots\}\]
of Lagrangian pinwheels, and a symplectic form on the same Godeaux
surface with the Mori sequence
\[M(4,1;33,10)=\{(4,1),(33,10),(227,69),(1556,473),(10665,3242),\ldots\}\]
of Lagrangian pinwheels. In the proof of Theorem \ref{thm:mainthm},
we will focus for convenience on {\em right mutations} and {\em
right initial antiflips}, but running the same arguments with left
mutations and left initial antiflips gives these other sequences.

\end{Remark}
\begin{Remark}
Our construction is a generalisation of the constructions by
Khodorovskiy \cite{KhoBalls}, Park-Park-Shin \cite{PPSBalls}, Owens
\cite{OwensBalls} and Park-Shin \cite{PSBalls}; we additionally keep
track of the symplectic form.

\end{Remark}
\begin{Remark}
It follows from the proof that the symplectic forms $\omega$ are
deformation equivalent to the forms representing the canonical class
$K$ coming from the canonical embedding, however our forms have
$[\omega]\neq K$. Since forms in the class $K$ admit only bounded
Lagrangian pinwheels, it is an interesting question to determine how
far one needs to deform $\omega$ away from the class $K$ before one
sees Mori sequences of pinwheels. We will discuss this in Section
\ref{sct:variation}, where we observe that our construction produces
unbounded pinwheels when the symplectic form crosses an affine
distance $\delta$ from the canonical class, where $\delta\geq 2$ is
an integer which shows up in the recursion formula for the Mori
sequence. It is not clear if this gap is an artefact of our
construction, and that there are unbounded pinwheels closer to the
canonical class, or if boundedness for pinwheels really persists in
some neighbourhood of the canonical class.

\end{Remark}
\subsection{Idea of proof}

The idea of the proof is to deform the symplectic form along a compact
codimension zero submanifold $U\subset X$. The submanifold $U$ has the
rational homology of $\cp{1}$ and $\partial U$ is a lens space. We
will exhibit a 1-parameter family of symplectic forms $\omega_t$ on
$U$ such that $(U,\omega_0)$ is negatively monotone and $(U,\omega_1)$
is positively monotone. The symplectic manifolds $(U,\omega_t)$ are
all symplectomorphic in a neighbourhood of $\partial U$, so the
deformation $\omega_t$ extends to a deformation of symplectic
structures on $X$ which is constant outside $U$. We call this
deformation an {\em initial antiflip} of the symplectic form.

We will then show that $(U,\omega_1)$ contains Mori sequences of
Lagrangian pinwheels. We prove this by giving an almost toric
structure on $(U,\omega_1)$ in which the pinwheels $L_{p_1,q_1}$ and
$L_{p_2,q_2}$ are visible surfaces, then performing an infinite
sequence of mutations\footnote{In the language of \cite{Symington}, a
mutation is a {\em branch move} which switches one of the branch cuts
in the almost toric structure for one pointing in the opposite
direction. The terminology {\em mutation} comes from the paper of
Galkin and Usnich \cite{GalkinUsnich}; the definition there is given
for the fan (rather than polytope) side of toric geometry.} to get
different almost toric structures on $U$ in which the pinwheels
$L_{p_i,q_i}$ and $L_{p_{i+1},q_{i+1}}$ are visible. We need to be
careful with our deformation of symplectic forms to ensure that there
is ``enough room'' in $U$ for an infinite sequence of mutations to be
performed.

This initial antiflip is related to the k2A 3-fold flip discovered by
Mori \cite{Mori} and further studied in \cite{HTU}. Roughly speaking,
the total space $\mathcal{X}$ of a \QG smoothing $\mathcal{X}\to\CC$
of a singular algebraic surface $\mathcal{X}_0$ can sometimes be
flipped to give a new \QG smoothing $\mathcal{X}^+\to\CC$ of a
different singular surface $\mathcal{X}^+_0$ without affecting any of
the smooth fibres: $\mathcal{X}_z\cong\mathcal{X}^+_z$ for $z\neq
0$. Since $\mathcal{X}_z$ and $\mathcal{X}^+_z$ arise from smoothing
different singularities, they contain the Milnor fibres of those
singularities. The same singular surface $\mathcal{X}^+_0$ can arise
when performing the flip of many different \QG smoothings
$\mathcal{X}$ of different singular surfaces $\mathcal{X}_0$ (indeed,
a whole Mori sequence of them).

This whole paper can be read as a symplectic topologist's guide to
\cite{HTU}, presenting those parts of that paper which can be cast
purely in terms of symplectic topology.

\subsection{Outline}

In Section \ref{sct:QHP}, we define {\em rational homology projective
lines} (QHPs) and construct toric orbifold QHPs, $V_\Pi$, from
polygons $\Pi$ which we call truncated wedges. We then construct
smooth QHPs, $U_\Pi$, as symplectic smoothings of these toric
QHPs. These manifolds are equipped with an almost toric fibration with
visible Lagrangian pinwheels.

In Section \ref{sct:mutations}, we study when the almost toric
fibrations on $U_\Pi$ can be mutated to give new almost toric
fibrations. This allows us to construct infinite sequences of visible
Lagrangian pinwheels corresponding to Mori sequences. In Section
\ref{sct:mori}, we define Mori sequences and summarise their
asymptotic behaviour. In Section \ref{sct:infinitemut} also discuss
when infinitely many mutations can be performed in a bounded region of
a truncated wedge.

In Section \ref{sct:mutfail}, we study those truncated wedges which
cannot be mutated and introduce a new operation which involves a
deformation of the symplectic form followed by a mutation. This leads
us to the {\em initial antiflip} of a symplectic form and its inverse,
the {\em flip}. The initial antiflip is a deformation of the
symplectic form, and, in Section \ref{sct:variation}, we discuss how
the cohomology class of $\omega$ varies along this deformation. In
Section \ref{sct:morilink}, we explain the link to Mori theory; in
Section \ref{sct:k1A}, we give the interpretation of k1A flips in our
setting; and, in Section \ref{sct:top}, we give a summary of how to
view the flip and antiflips topologically.

Finally, in Section \ref{sct:examples}, we give an algebro-geometric
recipe for constructing examples to which the theory applies and we
explain the examples stated in Theorem \ref{thm:mainthm}.

\subsection{Notation}

We will write $[b_1,\ldots,b_r]$ to mean both:
\begin{itemize}
\item a chain of spheres $C_1,\ldots,C_r$ which intersect according to the
graph

\begin{tikzpicture}[baseline=-0.5ex]
\node at (0,0) (C1) {$\bullet$};
\node at (1,0) (C2) {$\bullet$};
\node at (2,0) (C3) {$\cdots$};
\node at (3,0) (C4) {$\bullet$};
\node at (4,0) (C5) {$\bullet$};
\node at (C1) [below] {$C_1$};
\node at (C2) [below] {$C_2$};
\node at (C4) [below] {$C_{r-1}$};
\node at (C5) [below] {$C_r$};
\draw (C1.center) -- (C2.center) -- (C3);
\draw (C3) -- (C4.center) -- (C5.center);
\end{tikzpicture}
with self-intersections $C_i^2=-b_i$.
\item the continued fraction
\[[b_1,\ldots,b_r]=b_1-\frac{1}{b_2-\frac{1}{\cdots-\frac{1}{b_r}}}.\]
\end{itemize}
If we write $[b_{1,1},\ldots,b_{1,r_1}]-c-[b_{2,1},\ldots,b_{2,r_2}]$
we mean the chain
\[[b_{1,1},\ldots,b_{1,r_1},c,b_{2,1},\ldots,b_{2,r_2}],\] but where
we group together certain spheres which we wish to collapse down to a
singular point (or which have just arisen from resolving a singular
point).

\subsection{Acknowledgements}

JE is supported by EPSRC grant EP/P02095X/1. GU is supported by the
FONDECYT regular grant 1190066. The authors would like to thank: Ivan
Smith and Paul Hacking for helpful correspondence and conversations;
Nick Lindsay for helping us pinpoint a reference for the symplectic
suborbifold neighbourhood theorem; Daniele Sepe for pointing us
towards \cite{NgocSepe}; Anne-Sophie Kaloghiros for linguistic advice;
and an anonymous referee for their helpful comments.

\section{Rational homology projective lines}
\label{sct:QHP}

\begin{Definition}
A {\em rational homology projective line} (QHP) will mean a
4-dimensional manifold or orbifold $X$ with $H_*(X;\QQ)\cong
H_*(\cp{1};\QQ)$.

\end{Definition}
We will give a recipe for constructing symplectic QHPs as smoothings
of symplectic orbifold QHPs.

\subsection{Toric QHP-orbifolds: $V_\Pi$}

\subsubsection{Truncated wedges}

Given coprime integers $\Delta,\Omega$ with $0\leq \Omega<\Delta$, let
$\pi(\Delta,\Omega)$ denote the wedge
\[\pi(\Delta,\Omega):=\{(x,y)\in\RR^2\ :\ x\geq 0,\ \Delta y\geq
\Omega x\}.\]

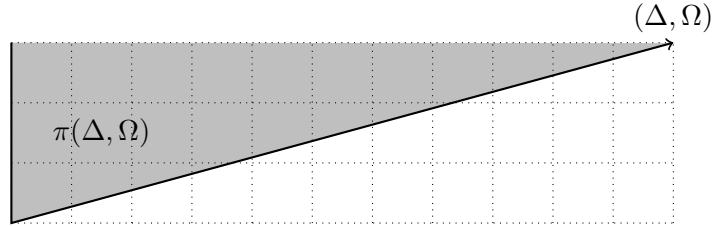
\begin{center}
\begin{tikzpicture}[scale=0.8]
\filldraw[draw=black,thick,->,fill=lightgray] (0,3) -- (0,0) -- (11,3);
\draw[step=1.0,black,thin,dotted] (0,0) grid (11,3);
\node at (1.5,1.5) {$\pi(\Delta,\Omega)$};
\node at (11,3) [above] {$(\Delta,\Omega)$};
\end{tikzpicture}

\end{center}
This is the moment polygon for a Hamiltonian torus action on the
cyclic quotient singularity\footnote{The cyclic quotient singularity
$\frac{1}{\Delta}(1,\Omega)$ is the quotient of $\CC^2$ by the action
of the group of $\Delta^{th}$ roots of unity given by
$\mu\cdot(x,y)=(\mu x,\mu^{\Omega}y)$.} $\frac{1}{\Delta}(1,\Omega)$.

Let $m,n$ be coprime integers with $n>0$ and let $h>0$ be a real
number. Consider the half-space $H_{m,n;h}=\{(x,y)\in\RR^2\ :\
mx+ny\geq h\}$ and the truncation
$\Pi=H_{m,n;h}\cap\pi(\Delta,\Omega)$.

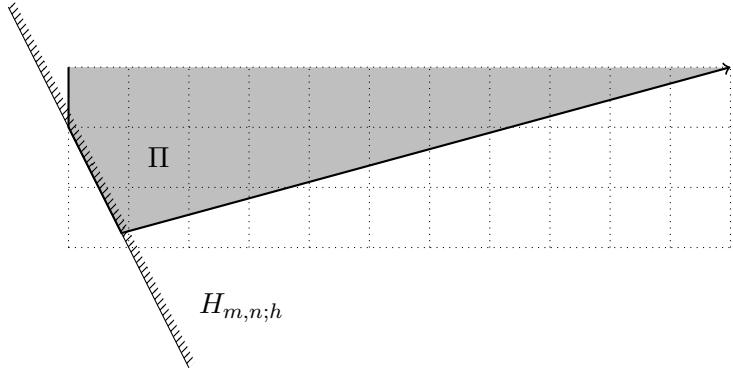
\begin{center}
\begin{tikzpicture}[scale=0.8]
\filldraw[draw=black,thick,->,fill=lightgray] (0,3) -- (0,2) -- (22/25,6/25) -- (11,3);
\draw[step=1.0,black,thin,dotted] (0,0) grid (11,3);
\node at (1.5,1.5) {$\Pi$};
\pattern[pattern=north east lines] (-1,4)--(-0.9,4.1)--(2.1,-1.9)--(2,-2)--cycle;
\draw (-1,4) -- (2,-2);
\node at (2,-1) [right] {$H_{m,n;h}$};
\end{tikzpicture}

\end{center}
This truncated wedge is the moment image of a partial resolution
$V_\Pi$ of the cyclic quotient singularity. The vertices $x_1$ and
$x_2$ of $\Pi$ are the images under the moment map of cyclic quotient
singularities (abusively, also called $x_1,x_2$) in $V_\Pi$; if $x_i$
has type $\frac{1}{P_i}(1,Q_i)$ then:
\begin{itemize}
\item $P_1=n$, $Q_1=-m\mod P_1$,
\item $P_2=m\Delta+n\Omega$, $Q_2=k\Delta+\ell\Omega\mod P_2$, where
$kn-\ell m=1$.
\end{itemize}
We will also abusively say that the vertices $x_i$ have type
$\frac{1}{P_i}(1,Q_i)$.

\begin{Definition}
We will say that a vertex of a polygon is a {\em Wahl vertex} if it
has type $\frac{1}{p^2}(1,pq-1)$ for some coprime integers $0\leq
q\leq p\neq 0$ (Wahl singularities are precisely the cyclic quotient
surface singularities of this type, see ({\cite[Remark
5.10]{LooijengaWahl}})). Below, $x_i$ will be a Wahl vertex of type
$\frac{1}{p_i^2}(1,p_iq_i-1)$.

\end{Definition}
\begin{Remark}\label{rmk:warning}
Note that we allow $(p,q)=(1,1)$ and $(p,q)=(1,0)$, both of which
represent a smooth point in $V_\Pi$. In order for our formulae below
to work out, we must only ever use $(1,0)$ for a smooth point $x_1$
and $(1,1)$ for a smooth point $x_2$. If you accidentally plug in
$p_1=q_1=1$ or $p_2=1$, $q_2=0$, then you will get the wrong
answers.

\end{Remark}
\subsubsection{Shear invariant}

Let $\Pi$ be a truncated wedge. Let $E_\Pi$ denote the edge between
$x_1$ and $x_2$ and let $C_\Pi\subset V_\Pi$ denote the corresponding
component of the toric boundary; $C_\Pi$ is a rational curve which
generates $H_2(V_{\Pi};\QQ)$.

\begin{Definition}
The {\em shear invariant} of $\Pi$ is defined to be the integer $c$
such that $\tilde{C}_{\Pi}^2=-c$, where $\tilde{C}_\Pi$ is the
proper transform of $C_\Pi$ in the minimal resolution
$\tilde{V}_\Pi\to V_\Pi$.

\end{Definition}
The reason for the name is visible in the standard moment polygon for
the total space of the line bundle $\mathcal{O}(-c)\to\cp{1}$, which
is a truncated wedge with shear invariant $c$, with the zero-section
(self-intersection $-c$) living over the compact edge:

\begin{center}
\begin{tikzpicture}
\filldraw[draw=black,thick,->,fill=lightgray] (0,1) -- (0,0) -- (1,0) -- (3,1);
\node at (3,1) [above] {$(1,c)$};

\end{tikzpicture}
\end{center}
\subsubsection{Constructing polygons}

\begin{Definition}\label{dfn:model}
Given a real number $a>0$ and integers $p_1,q_1,p_2,q_2,c$ such that
$0\leq q_1<p_1$, $0<q_2\leq p_2$ and $\gcd(p_i,q_i)=1$ for $i=1,2$,
define the polygon
\begin{align*}
\Pi(p_1,q_1,p_2,q_2,c,a):=\left\{(x,y)\in\RR^2\ :\ \right.&y\geq 0,\\
&p_1^2x\geq y(p_1(p_1-q_1)-1)\\
&\left.p_2^2(x-a)\leq y(cp_2^2-p_2q_2+1)\right\}.
\end{align*}

\end{Definition}
\begin{Remark}
As mentioned in Remark \ref{rmk:warning}, this definition does not
allow $(p_1,q_1)=(1,1)$ and $(p_2,q_2)=(1,0)$; rather, you should
use $(p_1,q_1)=(1,0)$ and $(p_2,q_2)=(1,1)$).

\end{Remark}
The polygon $\Pi:=\Pi(p_1,q_1,p_2,q_2,c,a)$ has:
\begin{itemize}
\item one horizontal compact edge $E_\Pi$ of affine length $a$,
\item two noncompact edges:
+ $R_1$, emanating from the origin and pointing in the direction
\[\vect{p_1(p_1-q_1)-1}{p_1^2}\]
+ $R_2$, emanating from the point $(a,0)$ and pointing in the
direction \[\vect{cp_2^2-p_2q_2+1}{p_2^2}.\]
\item Wahl vertices $x_i$ of type $\frac{1}{p_i^2}(1,p_iq_i-1)$,
\item shear invariant $c$.
\end{itemize}
We illustrate the polygon $\Pi$ below.

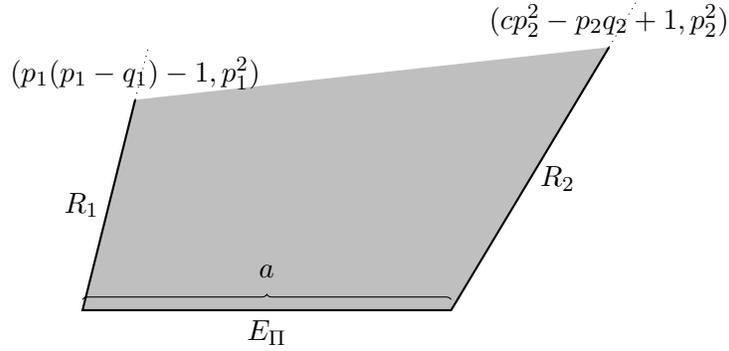
\begin{center}
\begin{tikzpicture}[scale=0.7]
\filldraw[draw=black,thick,fill=lightgray] (1,4) node [above] {$(p_1(p_1-q_1)-1,p_1^2)$} -- (0,0) -- (7,0) -- (10,5) node [above] {$(cp_2^2-p_2q_2+1,p_2^2)$};
\node at (0.5,2) [left] {$R_1$};
\node at (8.5,2.5) [right] {$R_2$};
\node at (3.5,0) [below] {$E_\Pi$};
\draw [decorate,decoration={brace,amplitude=2pt,raise=4pt},yshift=0pt]
(0,0) -- (7,0) node [sloped,midway,above=0.3] {$a$};
\draw[dotted] (1,4) -- (1.25,5);
\draw[dotted] (10,5) -- (10.6,6);
\end{tikzpicture}

\end{center}
This relates to our earlier description of polygons as truncations of
$\pi(\Delta,\Omega)$ in the following way:

\begin{Lemma}\label{lma:trunc}
Given a polygon $\Pi:=\Pi(p_1,q_1,p_2,q_2,c,a)$, define
\begin{align*}
\sigma(\Pi)&:=(c-1)p_1p_2+p_2q_1-p_1q_2,\\
\Delta(\Pi)&=p_1^2+p_2^2+\sigma(\Pi)p_1p_2,\\
\Omega(\Pi)&=p_1q_1+p_2q_2-1+\sigma(\Pi)p_2q_1-(c-1)p_2^2\mod\Delta(\Pi).
\end{align*}
If $\Delta(\Pi)>0$ then $\Pi$ is \(\ZZ\)-affine isomorphic to a
truncation of $\pi(\Delta(\Pi),\Omega(\Pi))$.
\end{Lemma}
\begin{proof}
We may apply the matrix
\[\matr{p_1^2}{1-p_1^2+p_1q_1}{p_1q_1-1}{1-p_1q_1+q_1^2}\] to $\Pi$
to move the edge $R_1$ so that it points in the direction $(0,1)$;
this moves $R_2$ into the direction
\[\vect{\Delta}{\Omega'}=\vect{p_1^2+p_2^2+\sigma
p_1p_2}{p_1q_1+p_2q_2-1+\sigma p_2q_1-(c-1)p_2^2}\] where
$\sigma=(c-1)p_1p_2+p_2q_1-p_1q_2$.

If $\Omega'=k\Delta+\Omega$, where $0\leq\Omega<\Delta$, then
shearing using the matrix $\matr{1}{0}{-k}{1}$ allows us to see
$\Pi$ as a truncation of the wedge $\pi(\Delta,\Omega)$. \qedhere

\end{proof}
\begin{Remark}
The number $\sigma(\Pi)$ from Lemma \ref{lma:trunc} has geometric
meaning: it is equal to $p_1p_2K_{V_\Pi}\cdot C_\Pi$. This means
that $V_\Pi$ is \(K\)-positive or \(K\)-negative if $\sigma(\Pi)$ is
positive or negative respectively. We will say that our polygon
$\Pi$ is {\em \(K\)-positive} or {\em \(K\)-negative} according to
the sign of $\sigma(\Pi)$.

\end{Remark}
\begin{Remark}\label{rmk:deltapos}
The condition $\Delta(\Pi)>0$ is equivalent to requiring that the
rays $R_1$ and $R_2$ do not intersect, which is necessary for $\Pi$
to be \(\ZZ\)-affine equivalent to a truncated wedge.

\end{Remark}
Instead of specifying $p_1,q_1,p_2,q_2,c$, we can equivalently specify
the chain $[b_{1,1},\ldots,b_{1,r_1}]-c-[b_{2,1},\ldots,b_{2,r_2}]$
where:
\begin{align*}
\tilde{C}_\Pi^2&=-c,\\
\frac{p_i^2}{p_iq_i-1}&=[b_{i,1},\ldots,b_{i,r_i}]
\end{align*}

\begin{Remark}
This chain of spheres
$[b_{1,1},\ldots,b_{1,r_1}]-c-[b_{2,1},\ldots,b_{2,r_2}]$ arises in
the minimal resolution of $\tilde{V}_\Pi\to V_\Pi$ as the preimage
of $C_\Pi$. Note that $\tilde{V}_\Pi$ is also toric; its moment
polygon $\tilde{\Pi}$ is obtained from $\Pi$ by a sequence of
truncations at non-Delzant vertices (see Figure
\ref{fig:minres}). With our conventions, in the minimal resolution
of a vertex of type $\frac{1}{P}(1,Q)$, the exceptional spheres with
self-intersections $-b_1,\ldots,-b_r$ with $P/Q=[b_1,\ldots,b_r]$
are encountered in that order as one moves {\em anticlockwise} around
the boundary of $\tilde{\Pi}$. Reversing the order corresponds to
replacing $Q$ by its multiplicative inverse modulo $P$ (if $P=p^2$,
$Q=pq-1$, this means replacing $q$ by $p-q$).

\end{Remark}
\begin{figure}
\begin{center}
\begin{tikzpicture}[scale=0.6]
\filldraw[draw=black,fill=lightgray,opacity=0.5] (0,4) -- (0,0) -- (10,4);
\filldraw[draw=black,fill=lightgray] (0,4) -- (0,1) -- (2,1) -- (5,2) -- (10,4);
\draw[step=1.0,black,thin,dotted] (0,0) grid (10,4);
\node at (2,3) {$\tilde{\pi}(P,Q)$};
\node at (0,1) {$\bullet$};
\node at (2,1) {$\bullet$};
\node at (5,2) {$\bullet$};
\node at (1,1) [above] {$-3$};
\node at (3.5,5/3) [above] {$-2$};
\end{tikzpicture}
\end{center}
\caption{The moment polygon $\tilde{\pi}(5,2)$ for the minimal resolution of $\frac{1}{5}(1,2)$ superimposed on the moment polygon $\pi(5,2)$ for the singularity. We have labelled the self-intersections of the curves in the exceptional locus. The continued fraction expansion of $\frac{5}{2}$ is $3-\frac{1}{2}$, and we see a \(-3\)-sphere and a \(-2\)-sphere as we move around the boundary of $\tilde{\pi}(5,2)$ anticlockwise.}
\label{fig:minres}
\end{figure}
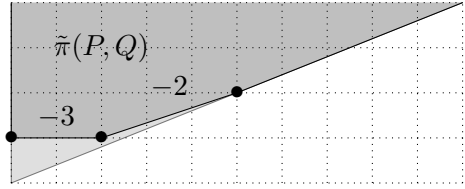

In terms of this chain there is a simple way to compute $\Delta$ and
$\Omega$:
\[\frac{\Delta}{\Omega}=[b_{1,1},\ldots,b_{1,r_1},c,b_{2,1},\ldots,b_{2,r_2}].\]

\subsection{Smooth, almost toric QHPs: $U_\Pi$}

Since $x_1$ and $x_2$ are Wahl singularities, we may symplectically
smooth these points, replacing them with symplectic rational homology
balls $B_{p_i,q_i}$, see \cite{SymingtonQBD}. This operation gives a
smooth symplectic QHP which we denote by $U_\Pi$.

\subsubsection{Almost toric structure}
\label{sct:almosttoric}

The operation of passing from $V_\Pi$ to $U_\Pi$ can be visualised by
means of an almost toric structure on $\Pi$: we perform {\em nodal
trades} at the two vertices of $\Pi$, introducing a branch cut at each
vertex.

\begin{Remark}
We briefly recall Symington's nodal trades \cite{Symington}. We can
modify the affine structure on $\pi(p^2,pq-1)$ by cutting from the
origin along a branch cut in the \((p,q)\)-direction to an interior
terminus $z$, and regluing the two sides using the affine monodromy
matrix $\matr{1+pq}{-p^2}{q^2}{1-pq}$. More precisely, we choose a
coorientation $(q,-p)$ of the branch cut and apply the affine
monodromy to tangent vectors as we cross the branch cut in the
direction of the coorientation (and its inverse if we cross in the
opposite direction). This modification is called a {\em nodal
trade}. The toric fibration on the Wahl singularity
$\frac{1}{p^2}(1,pq-1)$ deforms to give an {\em almost toric
fibration} on $B_{p,q}$. This almost toric fibration is a map from
$B_{p,q}$ to this modified affine surface whose general fibres are
Lagrangian tori; moreover, the affine structure on the base agrees
with the natural one given by local action-angle coordinates on a
Lagrangian fibration. Over the singular point $z$, there is a {\em
focus-focus} singular fibre (nodal torus); living over points in the
boundary there are circles.

\end{Remark}
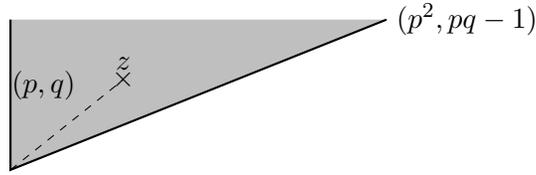
\begin{figure}
\begin{center}
\begin{tikzpicture}
\filldraw[draw=black,thick,fill=lightgray] (0,2) -- (0,0) -- (5,2);
\node at (1.5,1.2) {$\times$};
\node at (1.5,1.2) [above] {$z$};
\draw[dashed] (0,0) -- (1.5,1.2);
\node at (1,.8) [above left] {$(p,q)$};
\node at (5,2) [right] {$(p^2,pq-1)$};
\end{tikzpicture}
\end{center}
\caption{The wedge $\pi(p^2,pq-1)$ together with the branch cut for
performing a nodal trade.}
\end{figure}

To extend this construction to $\Pi$, we perform nodal trades at both
vertices, introducing two branch cuts, $B_1$ and $B_2$, joining the
vertices to interior points $z_1,z_2$. To determine in which direction
the branch cut $B_i$ is to be taken, one must use a \(\ZZ\)-affine
transformation to take a neighbourhood of the vertex in the model
$\pi(p_i^2,p_iq_i-1)$ to a neighbourhood of the vertex $x_i\in\Pi$;
the branch cut should be taken along the image of $(p,q)$ under this
transformation. In our model for $\Pi(p_1,q_1,p_2,q_2,c,a)$ from
Definition \ref{dfn:model}, the branch cuts $B_1$ and $B_2$ for the
nodal trades of are made in the directions:
\begin{itemize}
\item $(p_1-q_1,p_1)$ at vertex $x_1$;
\item $(cp_2-q_2,p_2)$ at vertex $x_2$.

\end{itemize}
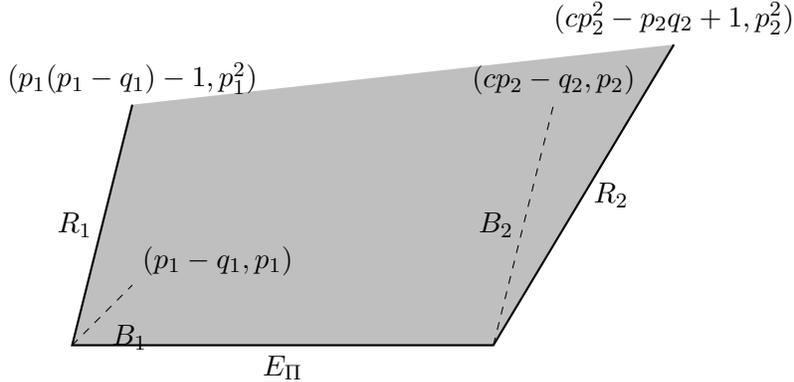
\begin{center}
\begin{tikzpicture}[scale=0.8]
\filldraw[draw=black,thick,fill=lightgray] (1,4) node [above] {$(p_1(p_1-q_1)-1,p_1^2)$} -- (0,0) -- (7,0) -- (10,5) node [above] {$(cp_2^2-p_2q_2+1,p_2^2)$};
\node at (0.5,0.5) [below right] {$B_1$};
\node at (7.5,2) [left] {$B_2$};
\node at (0.5,2) [left] {$R_1$};
\node at (8.5,2.5) [right] {$R_2$};
\node at (3.5,0) [below] {$E_\Pi$};
\draw[dashed] (0,0) -- (1,1) node [above right] {$(p_1-q_1,p_1)$};
\draw[dashed] (7,0) -- (8,4) node [above] {$(cp_2-q_2,p_2)$};
\end{tikzpicture}

\end{center}
The manifold $U_\Pi$ admits an almost toric fibration to this new
singular \(\ZZ\)-affine surface:
\begin{itemize}
\item over the points of the interior of $\Pi\setminus\{z_1,z_2\}$, we
have a Lagrangian torus fibre;
\item over $z_1,z_2$ there are a singular Lagrangian fibres (pinched
tori);
\item over each point of the boundary of $\Pi$ we have a circle; the
preimage of the whole boundary is a symplectic cylinder;
\item over the branch cut $B_i$ there lives\footnote{Surfaces like this
which project to lines in the base of an almost toric fibration are
called {\em visible surfaces} in \cite{Symington}.} a Lagrangian
disc \(L_i\) which becomes immersed \(p_i\)-to-\(1\) along its
boundary; this is called a {\em Lagrangian pinwheel}. A
neighbourhood of this pinwheel is the symplectic rational homology
ball $B_{p_i,q_i}$.

\end{itemize}
\begin{Remark}
In general, an almost toric structure can be specified by drawing an
{\em almost toric base diagram}, which is a decorated polygon with
focus-focus singularities and branch cuts indicated. The symplectic
4-manifold on which the almost toric structure lives is
determined\footnote{To reconstruct the symplectic manifold {\em
together with its almost toric structure}, one must make extra
choices at the singularities to determine the asymptotic behaviour
of the period lattice as one approaches the singularity. This was
first worked out by V\~{u} Ngoc \cite{Ngoc}; with this extra data,
the almost toric fibration is determined up to fibred
symplectomorphism {\cite[Theorem 4.60]{NgocSepe}}. Without this
data, the total space is still determined up to symplectomorphism.}
by an almost toric base diagram \cite{Symington}.

\end{Remark}
\subsubsection{The homology of $U_\Pi$}
\label{sct:homologyUPi}

It is easy to see that $U_\Pi$ is a QHP: since
$H_*(B_{p_i,q_i};\QQ)\cong H_*(B^4;\QQ)$, the rational homology of
$U_\Pi$ is isomorphic to that of the orbifold $V_\Pi$, which is
isomorphic to $H_*(\cp{1};\QQ)$. A generator for $H_2(U_\Pi;\QQ)$ can
be described explicitly as follows. The preimage of the edge $E_\Pi$
is a symplectic cylinder $C$ with area equal to the affine length $a$
of $E_\Pi$. Consider the singular chain
\(G_\Pi:=p_1p_2C-p_2L_1-p_1L_2\) (where \(L_i\) are the Lagrangian
discs living over the branch cuts). This is a cycle because \(\partial
L_1=p_1\partial C\) and \(\partial L_2=p_2\partial C\). The evaluation
of \([\omega]\) on its homology class can be computed by integrating
\(\omega\) over \(p_2L_1,p_1L_2,p_1p_2C\) separately, which yields
\(p_1p_2a\) (as the \(L_i\) are Lagrangian). Therefore $G_\Pi$
generates $H_2(U_\Pi;\QQ)$.

\section{Mutations}
\label{sct:mutations}

\subsection{Mutation of polygons}

\begin{Definition}
Suppose we equip a polygon $\Pi$ with the data of an almost toric
base diagram. Given a branch cut $B_z$ emanating from a focus-focus
singularity $z$, let $B'_z$ be the ray emanating from $z$ in the
opposite direction to $B_z$. We assume that $B'_z$ is also disjoint
from the other branch cuts. The line $B_z\cup B'_z$ cuts $\Pi$ into
two pieces $\Pi_{upper}$ and $\Pi_{lower}$ (where the coorientation
points into $\Pi_{upper}$). The {\em mutation} of $\Pi$ along $B_z$
is the polygon $\Pi_{upper}\cup A\Pi_{lower}$ (or, \(\ZZ\)-affine
equivalently, $A^{-1}\Pi_{upper}\cup \Pi_{lower}$), where $A$ is the
affine monodromy across the branch cut $B_z$. The mutated almost
toric base data is unchanged on $\Pi_{upper}$ and transformed by $A$
on $\Pi_{lower}$.

\end{Definition}
\begin{exm}
In Figure \ref{fig:mutpoly} we see a mutation of almost toric
structures on $\cp{2}$. The structure before mutation is obtained
from the standard toric structure by performing nodal trades at each
corner. The affine monodromy for the branch cut $B_z$ is
$A=\matr{0}{1}{-1}{2}$, which is the unique $A\in SL(2,\ZZ)$ which
satisfies both $A(1,1)=(1,1)$ (so it has $B_z$ as an eigendirection)
and $A(1,0)=(0,-1)$ (which means that, after mutation, the origin is
an interior point of a straight edge).

\end{exm}
\begin{figure}[htb]
\centering
\subfigure[Before mutation.]{
\begin{tikzpicture}[baseline=0ex,scale=0.8]
\filldraw[thick,fill=gray!50] (0,0) -- (6,0) -- (0,6) -- (0,0);
\draw[dashed] (0,0) -- (2,2) node {$\times$};
\node at (2,2) [above left] {$z$};
\node at (1,1) [above=2mm] {$B_z$};
\node at (1,3) [above] {$\Pi_{upper}$};
\node at (3,1) [below] {$\Pi_{lower}$};
\draw[dashed] (6,0) -- (4,1) node {$\times$};
\draw[dashed] (0,6) -- (1,4) node {$\times$};
\draw[dotted] (2,2) -- (3,3) node [pos=0.4,above=2mm] {$B'_z$};
\end{tikzpicture}
}\hspace{1cm}
\subfigure[After mutation.]{
\begin{tikzpicture}[baseline=0ex,scale=0.8]
\filldraw[thick,fill=gray!50] (0,6) -- (0,-6) -- (3,3) -- (0,6);
\draw[dotted] (0,0) -- (2,2) node {$\times$};
\node at (2,2) [above left] {$z$};
\node at (1,1) [above=2mm] {$B_z$};
\node at (1,3) [above] {$\Pi_{upper}$};
\node at (1,0) [below] {$A\Pi_{lower}$};
\draw[dashed] (0,-6) -- (1,-2) node {$\times$};
\draw[dashed] (0,6) -- (1,4) node {$\times$};
\draw[dashed] (2,2) -- (3,3) node [pos=0.4,above=2mm] {$B'_z$};
\end{tikzpicture}
}
\caption{An example of mutation between two almost toric base diagrams of $\cp{2}$ (coming from the \QG degeneration of $\cp{2}$ to $\mathbb{P}(1,1,4)$). The line $B_z$ is dashed in (a) and dotted in (b); the line $B'_z$ is dotted in (a) and dashed in (b). Dashed lines are branch cuts; dotted lines indicate linear continuations of branch cuts and are not part of the almost toric data.}\label{fig:mutpoly}
\end{figure}
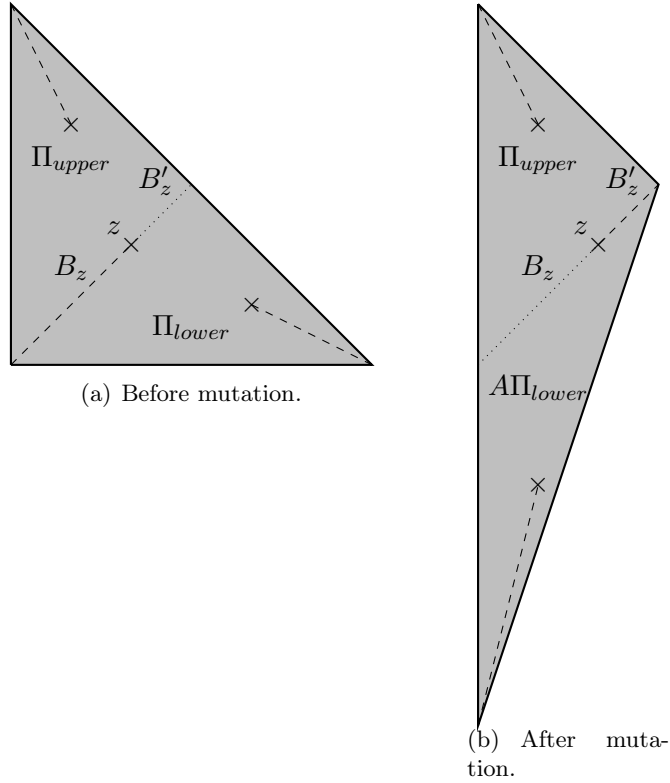

Though the polygons before and after mutation look very different,
this operation does not actually change the associated symplectic
manifold, as we will now prove:

\begin{Lemma}\label{lma:mutationdoesntchangeanything}
Let \(X_1\), \(X_2\) be almost toric symplectic 4-manifolds with
contractible almost toric base diagrams \(B_1\) and \(B_2\). Suppose
that \(B_1\) and \(B_2\) are related by a mutation. Then \(X_1\) and
\(X_2\) are symplectomorphic.
\end{Lemma}
\begin{proof}
Let \(f\colon X\to A\) be an almost toric fibration with focus-focus
fibres over the set \(A_{ff}\subset A\). Let \(\tilde{A}\) be the
universal cover of \(A\setminus A_{ff}\). Action-angle coordinates
allow us to define an integral affine structure on \(\tilde{A}\)
which descends to an integral affine structure on \(A\setminus
A_{ff}\). It is this integral affine structure which determines
\(X\) up to symplectomorphism (this follows by {\cite[Theorem
1.5]{Zung}} when the bases are contractible because then the
Lagrangian Chern class automatically vanishes).

An almost toric base diagram \(B\) can be constructed from \(A\) as
follows. By choosing branch cuts, pick a fundamental domain
\(B\subset\tilde{A}\) for the action of the deck group. Let
\(I\colon\tilde{A}\to\RR^2\) be the developing map for the integral
affine structure. The developing map does not descend to \(A\), but
its restriction to \(B\) can be considered as a ``branch'' of the
developing map on \(A\). The branch cuts form part of the boundary
of \(I(B)\subset\RR^2\); the branch cuts are identified by the
action of specified (integral affine) deck transformations of
\(\tilde{A}\), so to reconstruct \(A\) all we need is the immersed
polygon \(I(B)\) together with a collection of integral affine
transformations (``affine monodromies'') which identify pairs of
boundary components of \(I(B)\). This is precisely the data of an
almost toric base diagram.

We often (but not always) pick these branch cuts to point along the
eigenvectors of the deck transformations which pair them; if we do
this then any two branch cuts which are identified have the property
that their images under the developing map coincide, so that
\(I(B)\) is actually homeomorphic to \(A\). However, the branch cut
data is still important for reconstructing the integral affine
structure on \(A\): when you cross a branch cut, vectors normal to
the cut will still be affected by the affine monodromy. If branch
cuts are chosen along eigenlines of the monodromy then we say the
diagram is {\em eigensliced}.

Mutation is the operation on eigensliced almost toric base diagrams
which corresponds to changing branch cuts by 180 degrees. This
produces another eigensliced almost toric base diagram. Mutation
changes the branch cuts, and thereby affects the almost toric base
diagram, but does not affect the underlying integral affine manifold
\(A\). In particular, it does not affect the symplectomorphism type
of \(X\), which depends only on the integral affine structure of
\(A\). \qedhere

\end{proof}
\begin{Remark}
There is another closely related operation on almost toric base
diagrams which often accompanies a mutation, namely a {\em nodal
slide}. This is when a focus-focus point moves in the direction of
the eigenvector of its affine monodromy. This also leaves the
symplectomorphism type of \(X\) unaffected {\cite[Proposition
6.2]{Symington}}, but it does change the Lagrangian torus fibration
(whereas a mutation only changes the picture we draw to represent
the torus fibration).

\end{Remark}
\subsection{Mutability}

Mutation of polygons always makes sense, but it is possible that the
mutation of a truncated wedge is no longer a truncated wedge. We
therefore make the following definitions:

\begin{Definition}
Given the polygon $\Pi=\Pi(p_1,q_1,p_2,q_2,c,a)$ and the almost
toric structure with branch cuts $B_1$ and $B_2$, let $\bar{B}_i$
denote the semi-infinite ray through $x_i$ extending $B_i$. We say
that $\Pi$ is:
\begin{itemize}
\item {\em right-mutable} if $\bar{B}_1$ intersects the edge $R_2$,
\item {\em right-borderline} if $\bar{B}_1$ is parallel to $R_2$,
\item {\em right-immutable} otherwise.
\end{itemize}
Left-mutability is defined similarly. The {\em right mutation}
\end{Definition}
$\rmut(\Pi)$ is the mutation of $\Pi$ along $B_1$. For notational
convenience, we will focus entirely on right rather than left
mutation in what follows; indeed, one can reflect one's polygon in a
vertical line and always work with right mutation.

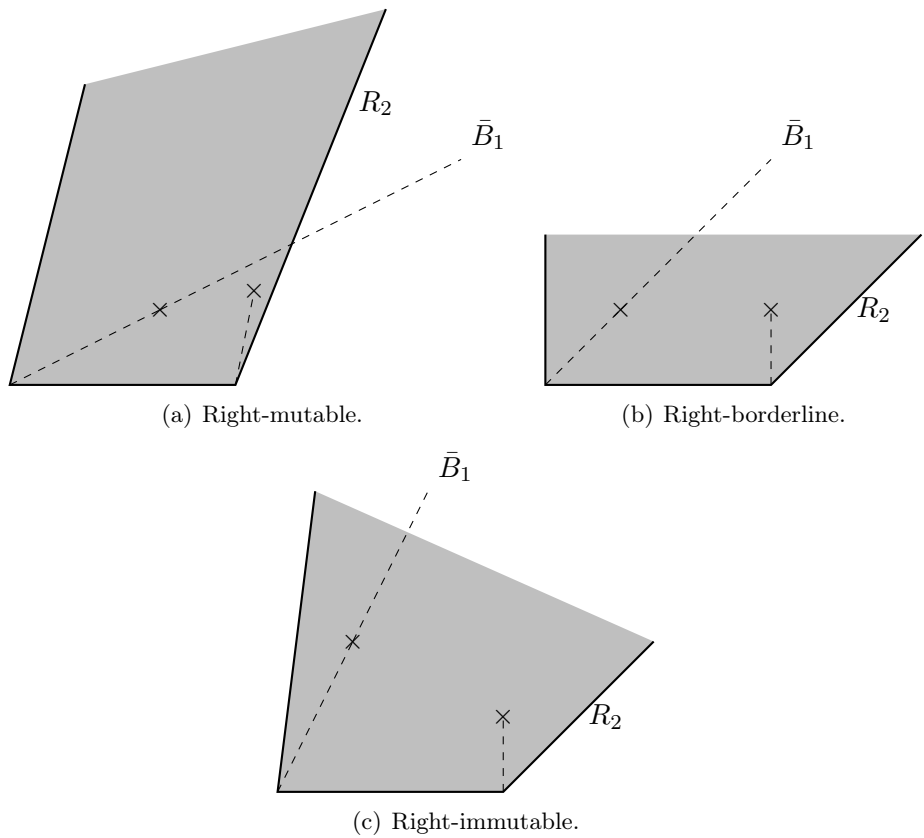
\begin{figure}
\centering
\subfigure[Right-mutable.]{
\begin{tikzpicture}[scale=1]
\filldraw[draw=black,thick,fill=gray!50] (1,4) -- (0,0) -- (3,0) -- (5,5);
\draw[dashed] (0,0) -- (2,1) node {$\times$};
\draw[dashed] (2,1) -- (6,3) node [above right] {$\bar{B}_1$};
\draw[dashed] (3,0) -- (3.25,1.25) node {$\times$};
\node at (4.5,3.75) [right] {$R_2$};
\end{tikzpicture}
}
\subfigure[Right-borderline.]{
\begin{tikzpicture}[scale=1]
\filldraw[draw=black,thick,fill=gray!50] (0,2) -- (0,0) -- (3,0) -- (5,2);
\draw[dashed] (0,0) -- (1,1) node {$\times$};
\draw[dashed] (1,1) -- (3,3) node [above right] {$\bar{B}_1$};
\draw[dashed] (3,0) -- (3,1) node {$\times$};
\node at (4,1) [right] {$R_2$};
\end{tikzpicture}
}
\subfigure[Right-immutable.]{
\begin{tikzpicture}[scale=1]
\filldraw[draw=black,thick,fill=gray!50] (0.5,4) -- (0,0) -- (3,0) -- (5,2);
\draw[dashed] (0,0) -- (1,2) node {$\times$};
\draw[dashed] (1,2) -- (2,4) node [above right] {$\bar{B}_1$};
\draw[dashed] (3,0) -- (3,1) node {$\times$};
\node at (4,1) [right] {$R_2$};
\end{tikzpicture}
}
\caption{Mutability of truncated wedges.}\label{fig:mutability}
\end{figure}

For our model polygon $\Pi(p_1,q_1,p_2,q_2,c,a)$, the affine monodromy
of $B_1$ (with its coorientation pointing to the left) is
\[A=\matr{1+p_1q_1-p_1^2}{(p_1-q_1)^2}{-p_1^2}{1-p_1q_1+p_1^2},\] and
the affine monodromy of $B_2$ (with its coorientation pointing to the
right) is
$\matr{cp_2^2-p_2q_2+1}{-(cp_2-q_2)^2}{p_2^2}{1+p_2q_2-cp_2^2}$.

Note that, for a right mutation, $AE_\Pi$ points in the (negative)
\(R_1\)-direction so $R_1\cup AE_\Pi$ is now a single edge of
$\rmut(\Pi)$. Indeed, $A$ is determined by this condition and the
condition that it has $B_1$ as an eigenray.

\begin{Lemma}
A polygon $\Pi=\Pi(p_1,q_1,p_2,q_2,c,a)$ is right-mutable if and
only if $c\leq 1$ and $\delta p_2-p_1>0$, where
$\delta=-\sigma(\Delta)$. As a consequence, mutability implies
$\sigma(\Pi)<0$. For left-mutability, we use the inequality $\delta
p_1-p_2>0$ instead.
\end{Lemma}
\begin{proof}
Note that if $c\leq 0$ then the invariant $\Delta(\Pi)$ from Lemma
\ref{lma:trunc} is negative, so we do not consider this case.

If $c\geq 2$ then the slope $\frac{p_2^2}{cp_2^2-p_2q_2+1}$ of $R_2$
is less than or equal to $1$. The slope $\frac{p_1}{p_1-q_1}$ (or
$1$ if $p_1=q_1=1$) of $\bar{B}_1$ is greater than or equal to $1$,
so these lines never intersect and the polygon is not right-mutable.

If $c=1$ then we have right-mutability if and only if the slope of
$\bar{B}_1$ is strictly less than the slope of $R_2$:
\[\frac{p_2^2}{p_2^2-p_2q_2+1}>\frac{p_1}{p_1-q_1}.\] This gives
\[p_2(p_1q_2-p_2q_1)-p_1>0,\] which is again equivalent to
$\delta p_2-p_1>0$. In particular, we see that this can only
hold if $\sigma(\Pi)$ is negative.

The criterion for left-mutability is proved similarly. \qedhere

\end{proof}
\begin{Remark}
Mutability also makes sense when $\Delta(\Pi)<0$, and we always get
mutability in both directions. However, these polygons are not
truncated wedges, so we ignore them.

\end{Remark}
\subsection{Effect of mutations}

The polygon $\rmut(\Pi)$ has vertices at $x'_1=Ax_2$ and at $x'_2$,
the point of intersection between $\bar{B}_1$ and $R_2$. Since the
type of a vertex is invariant under \(\ZZ\)-affine transformations, we
see that $x'_1$ has type $\frac{1}{p_2^2}(1,p_2q_2-1)$.

\begin{Remark}
Remembering our convention that a smooth point has $(p_1,q_1)=(1,0)$
or $(p_2,q_2)=(1,1)$ if it occurs on the left or on the right
respectively, one sees that this should switch under a mutation;
however, if $(p_2,q_2)=(1,1)$ then the polygon is not right-mutable,
so it is never an issue.

\end{Remark}
To identify the type of vertex $x'_2$ we need a recognition lemma:

\begin{Lemma}\label{lma:recog}
Suppose we have an edge $R$ of a polygon and a branch cut $B$
disjoint from $R$ whose semi-infinite extension $\bar{B}$ intersects
$R$. Make a \(\ZZ\)-affine transformation $M$ to put $R$ in the vertical
direction with the polygon on its right. If $-MB$ points in the
direction $(p,q+kp)$ with $0<q<p$ then the result of mutation along
$B$ will have a vertex of type $\frac{1}{p^2}(1,pq-1)$ at the point
of intersection between $B$ and $R$.
\end{Lemma}
\begin{proof}
The polygon $\pi(p^2,pq-1)$ equipped with a branch cut starting at
the origin and pointing in the $(p,q)$ direction can be mutated to
get the right half-space with a branch cut pointing out to infinity
in the \((p,q)\)-direction. Shearing this using matrices
$\matr{1}{0}{k}{1}$ gives the local models in the lemma, which will
then necessarily give (a shear of) the original polygon
$\pi(p^2,pq-1)$ upon mutation. (The sign in $-MB$ is because we
reverse the direction of the branch cut when we mutate). \qedhere

\end{proof}
\begin{Lemma}\label{lma:muteffect}
Let $\rmut(\Pi)$ be the right mutation of
$\Pi(p_1,q_1,p_2,q_2,1,a)$. Define:
\begin{gather}
\delta:=-\sigma(\Pi)=p_2q_1-p_1q_2,\nonumber\\
p_3:=\delta p_2-p_1,\qquad q_3:=\delta q_2-q_1.\nonumber
\end{gather}
Then:
\begin{itemize}
\item the affine length of $E_{\rmut(\Pi)}$ is $p_1a/p_3$;
\item the vertex $x'_2$ has type $\frac{1}{p_3^2}(1,p_3q_3-1)$ where
\item $\sigma(\rmut(\Pi))=\sigma(\Pi)=-\delta$.
\end{itemize}
\end{Lemma}
\begin{proof}
To find $x'_2$, we parametrise $B_1$ as
$(\tau_1(p_1-q_1),\tau_1p_1)$ (or $(\tau_1,\tau_1)$ if $p_1=q_1=1$)
and $R_2$ as $(a+\tau_2(p_2(p_2-q_2)+1),\tau_2p_2^2)$ and we see
this intersection occurs when
\begin{align*}
\tau_1&=\frac{p_2^2\tau_2}{p_1}\\
\tau_2&=\frac{p_1a}{\delta p_2-p_1},
\end{align*}
where $\delta=p_1q_2-p_2q_1$ (since $c=1$). After mutation, a
fraction $\tau_2$ of the affine length of $R_2$ becomes the edge
$E_{\rmut(\Pi)}$, so the affine length of this edge in the new
polygon is $\tau_2=ap_1/(\delta p_2-p_1)$.

To see what kind of vertex we get at $x'_2$ after a left mutation,
we can use the affine transformation
$M:=\matr{-p_2^2}{p_2(p_2-q_2)+1}{-p_2q_2-1}{q_2(p_2-q_2)+1}$ to put
the ray $R_2$ in the direction $(0,-1)$; this makes $R_2$ vertical
and puts $\Pi$ to the right of $R_2$ so we may apply Lemma
\ref{lma:recog} and compute $-MB_1=(p_3,q_3)$ to get the type
$\frac{1}{p_3^3}(1,p_3q_3-1)$ of $x'_2$. Since $B_1$ points in the
direction $(p_1-q_1,p_1)$, $-MB_1$ points in the direction
$(p_3,q_3)$ where $p_3=\delta p_2-p_1$ and $q_3=\delta q_2-q_1$.

For the final part of the lemma, $\sigma(\Pi)=p_2q_1-p_1q_2$ and
\begin{align*}
\sigma(\rmut(\Pi))&=p_3q_2-p_2q_3\\
&=(\delta p_2-p_1)q_2-p_2(\delta q_2-q_1)\\
&=p_2q_1-p_1q_2.
\end{align*}

\end{proof}
\subsection{Mori sequences}
\label{sct:mori}

\subsubsection{Definition of Mori sequences}

\begin{Definition}
Let $(p_1,q_1)$ and $(p_2,q_2)$ be pairs of positive integers with
$\gcd(p_i,q_i)=1$, $q_i\leq p_i$. Using the notation
\[[b_1,\ldots,b_r]=b_1-\frac{1}{b_2-\frac{1}{\cdots-\frac{1}{b_r}}}\]
for continued fractions, let
\[\frac{p_i^2}{p_iq_i-1}=[b_{i,1},\ldots,b_{i,r_i}],\] and suppose
that $\Delta=p_1^2+p_2^2+\sigma p_1p_2>0$ and that the rational
number
\[\frac{\Delta}{\Omega}:=[b_{1,1},\ldots,b_{1,r_1},1,b_{2,1},\ldots,b_{2,r_2}]\]
is well-defined (no division by zero). Let \[\delta=p_1q_2-p_2q_1\]
and suppose that $\delta>0$. The {\em Mori sequence}
$M(p_1,q_1;p_2,q_2)$ is the sequence of pairs $(p_i,q_i)$ extending
$(p_1,q_1),(p_2,q_2)$ and satisfying the recursions
\begin{align*}
p_{i+2}&=\delta p_{i+1}-p_i,\\
q_{i+2}&=\delta q_{i+1}-q_i.
\end{align*}

\end{Definition}
\subsubsection{Behaviour of Mori sequences}
\label{sct:mori}

If $(p_i,q_i)$ is a Mori sequence then we can recast the recursion
relation for the $p_i$ as a matrix equation
\[\vect{p_{i+1}}{p_{i+2}}=\matr{0}{1}{-1}{\delta}\vect{p_i}{p_{i+1}}.\]
This matrix $M=\matr{0}{1}{-1}{\delta}$ has eigenvalues
$\lambda_{\pm}=\frac{\delta\pm\sqrt{\delta^2-4}}{2}$; if $\delta\geq
2$ then these are real and satisfy $\lambda_-\lambda_+=1$. In this
case, there are eigenrays with slopes $\lambda_{\pm}$.

Repeated application of $M$ defines a discrete dynamical system on the
plane, and the behaviour of $(p_{i+1},p_{i+2})=M^i(p_1,p_2)$ under
repeated application of $M$ is indicated by the arrows in the
figure. This behaviour separates into three distinct regions,
separated by the eigenrays:
\begin{itemize}
\item In the region $p_2>\lambda_+ p_1$, the Mori sequence is increasing
and the ratio $p_{i+1}/p_i$ tends to $\lambda_+$ from above.
\item In the region $p_2<\lambda_- p_1$, the Mori sequence is decreasing
and terminates when $M^i(p_1,p_2)$ leaves the positive quadrant.
\item In the region between the two eigenrays, the Mori sequence
decreases, reaches a minimum, then increases again. It does not
terminate in either direction.
\end{itemize}
Note that $(p_1,p_2)$ lives in the region between the eigenrays if and
only if $\Delta(\Pi)=p_1^2+p_2^2+\delta p_1p_2$ is negative. Recall
from Lemma \ref{lma:trunc} and Remark \ref{rmk:deltapos} that
$\Delta(\Pi)>0$ for all truncated wedges, so we find ourselves
automatically in the situation where our Mori sequence is increasing
or decreasing (if $p_1<p_2$ or $p_2<p_1$ respectively).

\begin{center}
\begin{tikzpicture}
\draw (0,5) -- (0,0) -- (5,0);
\draw[dashed] (0,0) -- (60:7) node [pos=0.85,sloped,above] {slope $\lambda_+$};
\draw[dashed] (0,0) -- (30:7) node [pos=0.85,sloped,below] {slope $\lambda_-$};
\draw[thick,->] (0,2) to[out=0,in=-120] (61:5);
\draw[thick,->] (0,3) to[out=0,in=-120] (62:5);
\draw[thick,->] (29:5) to[out=-150,in=90] (2,0);
\draw[thick,->] (28:5) to[out=-150,in=90] (3,0);
\draw[thick,->] (31:5) to[out=-150,in=-45] (2,2) to[out=135,in=-120] (59:5);
\node at (0,5) [left] {$p_2$};
\node at (5,0) [below] {$p_1$};

\end{tikzpicture}
\end{center}
\subsection{Infinite mutability}
\label{sct:infinitemut}

\begin{Definition}
We say that a \(K\)-negative polygon $\Pi$ is {\em infinitely
right-mutable} if $\rmut^j(\Pi)$ is right-mutable for
$j=0,1,\ldots$. From the previous subsection, this is equivalent to
$\delta\geq 2$, $p_1\leq p_2$.

\end{Definition}
If $\Pi$ is infinitely right-mutable, then, by Lemma
\ref{lma:muteffect}, we obtain a sequence of mutations
$\Pi(p_i,q_i,p_{i+1},q_{i+1},1,a_i)$ where $(p_i,q_i)$ is a Mori
sequence $M(p_1,q_1;p_2,q_2)$ ($\delta=-\sigma(\Pi)$).

By construction, the symplectic manifold $U_{\rmut^{j-1}(\Pi)}$
contains Lagrangian pinwheels $L_{p_j,q_j}$ and $L_{p_{j+1},q_{j+1}}$
as visible surfaces in its almost toric fibration, see Section
\ref{sct:almosttoric}. The manifolds $U_{\rmut^{j-1}(\Pi)}$ and
$U_\Pi$ are symplectomorphic by Remark
\ref{rmk:mutationdoesntchangeanything}, since their almost toric
structures are related by mutations. We summarise this in the
following corollary.

\begin{Corollary}
Let $\Pi=\Pi(p_1,q_1,p_2,q_2,1,a)$ be an infinitely mutable
polygon. The symplectic manifold $U_\Pi$ contains Lagrangian
pinwheels $L_{p_i,q_i}$ where $(p_i,q_i)$ is the Mori sequence
$M(p_1,q_1;p_2,q_2)$ with $\delta=-\sigma(\Pi)=p_1q_2-p_2q_1$.

\end{Corollary}
In practice, we are looking for these Mori sequences of pinwheels in
{\em compact} symplectic manifolds, so it is important that we can
perform the sequence of mutations in a compact subdomain of
$U_\Pi$. To that end, we introduce some new notation:

\begin{Definition}\label{dfn:bddpolygon}
Given a truncated wedge $\Pi=\Pi(p_1,q_1,p_2,q_2,c,a)$ and two
positive real numbers $\ell_1,\ell_2$, let $y_i\in R_i$ be the
unique point on the ray \(R_i\) at a distance $\ell_i$ from $x_i$,
for $i=1,2$. Define $\Pi_{\ell_1,\ell_2}$ to be the convex hull of
$x_1,x_2,y_1,y_2$. Let $V_{\Pi}(\ell_1,\ell_2)$ (respectively
$U_{\Pi}(\ell_1,\ell_2)$) be the preimage of $\Pi_{\ell_1,\ell_2}$
under the moment map (respectively almost toric fibration).

\end{Definition}
The manifold $U_{\Pi}(\ell_1,\ell_2)$ is a compact symplectic manifold
whose boundary is a lens space $L(\Delta,\Omega)$ of contact-type. The
diffeomorphism type is independent of the parameters
$a,\ell_1,\ell_2$, but these are important for the symplectic
structure.

\begin{Lemma}\label{lma:sufficientlysmall}
Let $\Pi^-=\Pi(p_1,q_1,p_2,q_2,1,a^-)$ be an infinitely
right-mutable \(K\)-negative polygon and let $\ell_1,\ell_2$ be
positive real numbers. If $\ell_2>\frac{a^-}{\lambda_+^{2}-1}$ then
the right mutations of $\Pi^-$ may be performed inside the
subpolygon $\Pi^-_{\ell_1,\ell_2}$. As a consequence,
$U_{\Pi^-}(\ell_1,\ell_2)$ contains an infinite Mori sequence
$M(p_1,q_1;p_2,q_2)$ of Lagrangian pinwheels.
\end{Lemma}
\begin{proof}
As we can see in Figure \ref{fig:muteats}, each mutation we perform
``eats up'' a certain amount of the affine length $\ell_2$ of $R_2$:
by Lemma \ref{lma:muteffect}, the first mutation uses
$a_1:=a^-p_1/p_3$ and the $k^{th}$ mutation uses
$a_k:=a_{k-1}p_k/p_{k+2}$. Therefore, in total, to perform
arbitrarily many mutations of this subpolygon, we need $\ell_2$ to
be at least
\[a^-\frac{p_1}{p_3}\left(1+\frac{p_2}{p_4}\left(1+\frac{p_3}{p_5}\left(1+\cdots\right)\right)\right).\]
By the discussion in Section \ref{sct:mori}, since
$\frac{p_2}{p_1}>\lambda_+$, the sequence of quotients
$\frac{p_i}{p_{i+1}}$ is increasing and its limit is $\lambda_-$;
likewise, the sequence
$\frac{p_i}{p_{i+2}}=\frac{p_i}{p_{i+1}}\frac{p_{i+1}}{p_{i+2}}$ is
increasing and its limit is $\lambda_-^2$. Therefore, the infinite
sum is bounded from above by
\[a^-\lambda_-^2\left(1+\lambda_-^2\left(1+\cdots\right)\right)
=\frac{a^-\lambda_-^2}{1-\lambda_-^2}=\frac{a^-}{\lambda_+^2-1},\]
as required. \qedhere

\end{proof}
\begin{figure}
\begin{center}
\begin{tikzpicture}[scale=1.3]
\filldraw[draw=black,thick,fill=gray!60] (0,2) -- (0,0.3) -- (0.3,0.1) -- (5,2);
\draw[dashed,thick] (0,0.3) -- (28/15,11/15);
\draw [decorate,decoration={brace,amplitude=2pt,mirror,raise=4pt},yshift=0pt]
(0,0.3) -- (0.3,0.1) node [sloped,midway,below=0.1] {$a_k$};
\draw [decorate,decoration={brace,amplitude=2pt,raise=4pt},yshift=0pt]
(0,0.3) -- (0,2) node [sloped,midway,above=0.1] {$\ell_1$};
\draw [decorate,decoration={brace,amplitude=10pt,mirror,raise=4pt},yshift=0pt]
(0.3,0.1) -- (5,2) node [sloped,midway,below=0.6] {$\ell_2$};
\node at (0,0.3) {$\bullet$};
\node at (0.3,0.1) {$\bullet$};
\filldraw[draw=black,thick,fill=gray!60] (0,2-2.5) -- (0,0.2-2.5) -- (28/15,11/15-2.5) -- (5,2-2.5);
\draw[dashed,thick] (0,0.3-2.5) -- (28/15,11/15-2.5);
\draw [decorate,decoration={brace,amplitude=10pt,mirror,raise=4pt},yshift=0pt]
(0,0.2-2.5) -- (28/15,11/15-2.5) node [sloped,midway,below=0.6] {$a_{k+1}=\frac{a_kp_k}{p_{k+2}}$};
\draw [decorate,decoration={brace,amplitude=2pt,raise=4pt},yshift=0pt]
(0,0.2-2.5) -- (0,2-2.5) node [sloped,midway,above=0.1] {$\ell_1+a_k$};
\draw [decorate,decoration={brace,amplitude=10pt,mirror,raise=4pt},yshift=0pt]
(28/15,11/15-2.5) -- (5,2-2.5) node [sloped,midway,below=0.6] {$\ell_2-a_{k+1}$};
\node at (0,0.2-2.5) {$\bullet$};
\node at (28/15,11/15-2.5) {$\bullet$};
\draw[->,thick] (5.5,1) to[out=0,in=0] (5.5,0.5-2.5);
\node at (7.4,1-1.75) {$k^{th}$ mutation};
\end{tikzpicture}
\end{center}
\caption{A mutation eats up available affine length. Before the
mutation, the right-hand edge $R_2$ has affine length $\ell_2$;
after mutation, it has lost affine length $a_{k+1}$.}
\label{fig:muteats}
\end{figure}
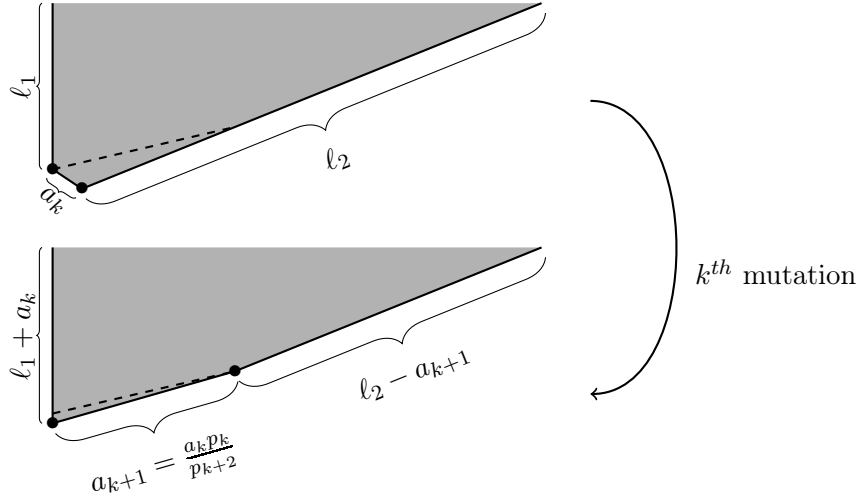

\section{When mutation fails}
\label{sct:mutfail}

\subsection{Immutability: flips and the initial antiflip}
\label{sct:flips}

Suppose we have a right-immutable polygon $\Pi$. We can make a
symplectic deformation $(U_t,\omega_t)$ of $U_\Pi$ and a deformation
of the almost toric structure to put us into a situation where the
mutation can be performed. We will show this by giving a family of
almost toric base diagrams $\Pi_t$ (which determine the symplectic
manifolds $U_t$). See Figure \ref{fig:initialantiflip} for an
illustration of this deformation.

\begin{enumerate}
\item We first perform the right mutation along $B_1$ (as $\Pi$ is
immutable, this will not be a truncated wedge: see Figure
\ref{fig:initialantiflip}). This replaces $B_1$ with an opposite
branch cut $B'_1$.
\item Pick a smooth path $\gamma\colon[0,1]\to\Pi$ such that:
\begin{itemize}
\item $\gamma(0)=z_1$,
\item $\gamma(t)\not\in B_2$ for all $t\in[0,1]$,
\end{itemize}
Let $B_1(t)$ (respectively $B'_1(t)$) be the ray pointing the
direction of $B_1$ (respectively $B'_1$) and emanating from
$\gamma(t)$. Assume that $\gamma(1)$ is sufficiently far to the
right so that $B_1(1)$ intersects $R_2$ at some point $x$.
\item When we perform a mutation along $B'_1(1)$, we therefore obtain a
new truncated wedge having $x$ as a vertex.

\end{enumerate}
\begin{Remark}
Note that this is a continuous deformation of symplectic manifolds:
although it involves steps which look discrete (mutations) these
steps do not affect the symplectomorphism type of \(U_t\) (see Lemma
\ref{lma:mutationdoesntchangeanything}). We are simply choosing to
draw pictures using different branch cuts at different stages of the
deformation, as this allows us to highlight different aspects of the
geometry.

\end{Remark}
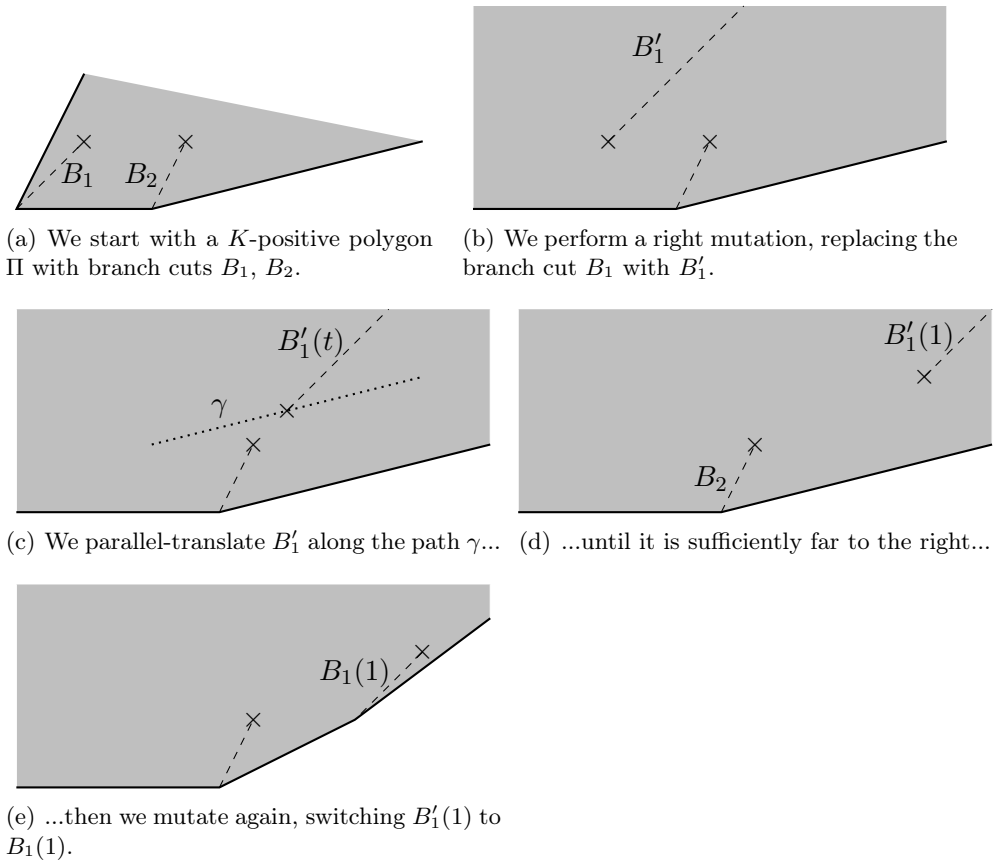
\begin{figure}
\subfigure[We start with a \(K\)-positive polygon $\Pi$ with branch cuts $B_1$, $B_2$.]
{
\begin{tikzpicture}[scale=0.9]
\filldraw[draw=black,thick,fill=gray!50] (1,2) -- (0,0) -- (2,0) -- (6,1);
\draw[dashed] (0,0) -- (1,1);
\draw[dashed] (2,0) -- (2.5,1);
\node at (1,1) {$\times$};
\node at (2.5,1) {$\times$};
\node at (0.5,0.5) [right] {$B_1$};
\node at (2.25,0.5) [left] {$B_2$};
\end{tikzpicture}
}
\hfill
\subfigure[We perform a right mutation, replacing the branch cut $B_1$ with $B'_1$.]
{
\begin{tikzpicture}[scale=0.9]
\filldraw[draw=white,fill=gray!50] (-1,3) -- (-1,0) -- (0,0) -- (2,0) -- (6,1) -- (6,3);
\draw[thick,black] (-1,0) -- (0,0) -- (2,0) -- (6,1);
\draw[dashed] (1,1) -- (3,3);
\draw[dashed] (2,0) -- (2.5,1);
\node at (1,1) {$\times$};
\node at (2.5,1) {$\times$};
\node at (2,2) [above left] {$B'_1$};
\end{tikzpicture}
}
\hfill
\subfigure[We parallel-translate $B'_1$ along the path $\gamma$...]
{
\begin{tikzpicture}[scale=0.9]
\filldraw[draw=white,fill=gray!50] (-1,3) -- (-1,0) -- (0,0) -- (2,0) -- (6,1) -- (6,3);
\draw[thick,black] (-1,0) -- (0,0) -- (2,0) -- (6,1);
\draw[dotted,thick] (1,1) -- (5,2);
\node at (2,1.25) [above] {$\gamma$};
\draw[dashed] (3,1.5) -- (4.5,3);
\draw[dashed] (2,0) -- (2.5,1);
\node at (3,1.5) {$\times$};
\node at (2.5,1) {$\times$};
\node at (1+3,1+1.5) [left] {$B'_1(t)$};
\end{tikzpicture}
}
\hfill
\subfigure[...until it is sufficiently far to the right...]
{
\begin{tikzpicture}[scale=0.9]
\filldraw[draw=white,fill=gray!50] (-1,3) -- (-1,0) -- (0,0) -- (2,0) -- (6,1) -- (6,3);
\draw[thick,black] (-1,0) -- (0,0) -- (2,0) -- (6,1);
\draw[dashed] (5,2) -- (6,3);
\draw[dashed] (2,0) -- (2.5,1);
\node at (5,2) {$\times$};
\node at (2.5,1) {$\times$};
\node at (2.25,0.5) [left] {$B_2$};
\node at (0.6+5,0.6+2) [left] {$B'_1(1)$};
\end{tikzpicture}
}
\hfill
\subfigure[...then we mutate again, switching $B'_1(1)$ to $B_1(1)$.]
{
\begin{tikzpicture}[scale=0.9]
\filldraw[draw=white,fill=gray!50] (-1,3) -- (-1,0) -- (0,0) -- (2,0) -- (4,1) -- (6,2.5) -- (6,3);
\draw[thick,black] (-1,0) -- (0,0) -- (2,0) -- (4,1) -- (6,2.5);
\draw[dashed] (5,2) -- (4,1);
\draw[dashed] (2,0) -- (2.5,1) node {$\times$};
\node at (5,2) {$\times$};
\node at (4,1.7) {$B_1(1)$};
\end{tikzpicture}
}
\caption{A cartoon of an initial antiflip.}\label{fig:initialantiflip}
\end{figure}

\begin{Definition}\label{dfn:flip}
Suppose we have a \(K\)-positive polygon
\[\Pi^+(a^+):=\Pi(p_0,q_0,p_1,q_1,c,a^+).\] Suppose that
$\Pi^-(a^-)$ is the result of performing the aforementioned
operations to $\Pi^+(a^+)$, where $a^-$ is the affine length of the
compact edge in the truncated wedge at the end of the process. We
call $\Pi^-(a^-)$ the {\em initial right-antiflip polygon} of
$\Pi^+(a^+)$ with parameter $a^-$ (initial left-antiflip is defined
in the obvious way). We call the symplectic manifold
$(U_{\Pi^-(a^-)},\omega_1)$ the {\em initial antiflip of the
symplectic form with parameter $a^-$}. We will omit the $a^{\pm}$
when it is unimportant to the discussion.

\end{Definition}
\begin{Remark}
The reverse procedure, in which we begin with a left-immutable
\(K\)-negative polygon and follow the same steps to force a left
mutation, is called the {\em flip}.

\end{Remark}
The parameter $a^-$ may be chosen freely by picking $\gamma$ suitably;
however, when we work with a bounded subset
$\Pi^+_{\ell_1,\ell_2}\subset\Pi^+$ as in Definition
\ref{dfn:bddpolygon}, we will not have complete freedom and $a^-$ will
need to be chosen sufficiently small. Namely, after an initial
antiflip, $\Pi^+(a^+)_{\ell_1,\ell_2}$ is replaced by
$\Pi^-(a^-)_{\ell_1+a^+,\ell_2-a^-}$, so we need $a^-<\ell_2$. If we
wish additionally to ensure infinite right-mutability within this
bounded polygon, we need the stronger inequality
\[\ell_2-a^->\frac{a^-}{\lambda_+^{2}-1},\] by Lemma
\ref{lma:sufficientlysmall}. This can also be achieved by picking
$a^-$ sufficiently small. We deduce the following corollary:

\begin{Corollary}\label{cor:infinitely}
Let $\Pi^+(a^+)=\Pi(p_1,q_1,p_2,q_2,c,a^+)$ be a \(K\)-positive
truncated wedge whose initial antiflip
$\Pi^-(a^-)=\Pi(p_1,q'_1,p_2,q_2,1,a^-)$ is infinitely right-mutable
(the numbers $q'_1,p_2,q_2$ will be defined in Lemma
\ref{lma:antiflipeffect}). If we are given $\ell_1,\ell_2>0$, then
there exists a constant $C>0$ such that, for all $0<a^-\leq C$, the
full Mori sequence of right mutations can be performed on
$\Pi^-(a^-)_{\ell_1+a^+,\ell_2-a^-}$. In particular, the initial
antiflip of the symplectic form with parameter $a^-$ in the range
$(0,C]$ admits a Mori sequence $M(p_1,q_1;p_2,q_2)$ of Lagrangian
pinwheels.

\end{Corollary}
\subsection{Numerology of the initial antiflip}

The following lemma is proved using Lemma \ref{lma:recog}; its proof
is very similar to Lemma \ref{lma:muteffect}, and we omit it:

\begin{Lemma}\label{lma:antiflipeffect}
Suppose we have a \(K\)-positive polygon
\[\Pi^+:=\Pi(p_0,q_0,p_1,q_1,c,a^+).\] Let
\[\delta=\sigma(\Pi)=(c-1)p_0p_1+p_1q_0-p_0q_1.\] Given a positive
real number $a^->0$, the initial antiflip polygon $\Pi^-(a^-)$ is
\(\ZZ\)-affine isomorphic to the polygon
\[\Pi^-(a^-)=\Pi(p_1,q'_1,p_2,q_2,1,a^-),\] where:
\begin{align*}
q'_1&=\begin{cases}0\mbox{ if }p_1=q_1=1\\ q_1\mbox{ otherwise,}\end{cases}\\
p_2&:=\delta p_1+p_0,\\
q_2&:=\frac{\delta+p_2q_1}{p_1}.
\end{align*}

\end{Lemma}
The following lemma is easy to check using Lemma \ref{lma:trunc} and
the definitions of $p_2,q_2$:

\begin{Lemma}
The initial antiflip polygon $\Pi^-$ is a left-immutable,
\(K\)-negative polygon with
\[\sigma(\Pi^-)=-\delta,\quad\Delta(\Pi^+)=\Delta(\Pi^-),\quad\Omega(\Pi^+)=\Omega(\Pi^-).\]
Consequently, both $\Pi^+$ and $\Pi^-$ are truncations of the same
wedge $\pi(\Delta,\Omega)$.

\end{Lemma}
\subsection{Variation of the cohomology class $[\omega_t]$}
\label{sct:variation}

Each almost toric base diagram in the family $\Pi_t$ from Section
\ref{sct:flips} determines a symplectic manifold, so we get a
symplectic deformation $(U_t,\omega_t)$. The de Rham cohomology group
$H^2_{dR}(U_t)$ is one-dimensional, so the cohomology class
$[\omega_t]$ is determined by its integral over some
fixed\footnote{i.e. constant with respect to the Gauss-Manin
connection.} homology class.

We use as our fixed class the unique class $G_t\in H_2(U_t;\RR)$ such
that $K_{U_t}\cdot G_t=\delta$. Since $K_{U_t}$ is an integral class,
this means that $G_t$ is also an integral class, hence constant in the
family. Recall the class $G_{\Pi}$ from Section \ref{sct:homologyUPi}:
\begin{itemize}
\item When $t=0$, we know that $K_{U_{\Pi^+}}\cdot
G_{\Pi^+}=\sigma(\Pi^+)=\delta$, so take $G_0=G_{\Pi^+}$.
\item When $t=1$, we know that $K_{U_{\Pi^-}}\cdot
G_{\Pi^-}=\sigma(\Pi^-)=-\delta$, so take $G_1=-G_{\Pi^-}$.
\end{itemize}
We know that $\int_{G_{\Pi^+}}\omega_0=p_0p_1a^+$ and
$\int_{G_{\Pi^-}}\omega_1=p_1p_2a^{-}$. Therefore, at the level of
cohomology classes $[\omega_t]$, the deformation of $\omega_t$ gives a
path in $H^2_{dR}(U)=\RR$ from $a^+p_0p_1$ to $-a^-p_1p_2$. In
particular, at some point in this path \(\omega_t\) is exact (at this
point the edge has length zero, so \(G_\Pi\) consists of two multiples
of Lagrangian discs sharing a common circle boundary).

The cohomology $H^2_{dR}(U)$ inherits a \(\ZZ\)-affine structure from
its isomorphism with $H^2(U;\ZZ)\otimes\RR$, so there is an intrinsic
notion of affine distances $d_{aff}$ along lines of rational
slope. For surfaces of general type, we use this to give an estimate
on how far one needs to deform $[\omega]$ away from the canonical
class before one gets unbounded Mori sequences of Lagrangian pinwheels
using our antiflip-and-mutate construction:

\begin{Lemma}\label{lma:gap}
Let $\Pi^+=\Pi(p_0,q_0,p_1,q_1,c,a^+)$ be a \(K\)-positive truncated
wedge whose initial antiflip polygon is infinitely right-mutable (so
$\delta=-\sigma(\Pi^+)=\sigma(\Pi^-)\geq 2$). Suppose that a compact
$U_{\Pi^+}(\ell_1,\ell_2)$ embeds symplectically into a symplectic
manifold $(X,\omega)$ with $[\omega]=K_X$. Let $\omega_t$ be the
initial antiflip deformation of the symplectic form on $X$ along the
submanifold $U_{\Pi^+}(\ell_1,\ell_2)$ with parameter $a^-$. Then
there exists a constant $\epsilon>0$ such that $(X,\omega_t)$
contains a Mori sequence of Lagrangian pinwheels when
$d_{aff}([\omega_0],[\omega_t])\in(\delta,\delta+\epsilon]$.
\end{Lemma}
\begin{proof}
By Corollary \ref{cor:infinitely}, there is a constant $C>0$ such
that the initial antiflip with parameter $a^-\in(0,C]$ contains a
Mori sequence of Lagrangian pinwheels. Therefore $(X,\omega_t)$
contains a Mori sequence of Lagrangian pinwheels whenever
$\int_{G_t}\omega_t\in[-Cp_1p_2,0)$. Let $t_0$ and $t_1$ be the
times such that $\int_{G_{t_0}}\omega_{t_0}=0$ and
$\int_{G_{t_1}}\omega_{t_1}=-Cp_1p_2$ (we have $t_0<t_1$ since the
\(\omega_t\)-area of $G_t$ is decreasing in $t$).

Since $[\omega]=K_X$, the number $a^+p_0p_1$ is integral (it is the
canonical class evaluated on the generator $G_{\Pi^+}\in H_2(U;\ZZ)$
from Section \ref{sct:homologyUPi}). In fact,
$a^+p_0p_1=\delta=-\sigma(\Pi^+)=\sigma(\Pi^-)$. Therefore,
$d_{aff}(\omega_0,\omega_{t_0})=\delta$ and
$d_{aff}(\omega_0,\omega_{t_1})=\delta+Cp_1p_2$, so we take
$\epsilon=Cp_1p_2$. \qedhere

\end{proof}
\subsection{Link with Mori theory}
\label{sct:morilink}

Given a \(K\)-positive polygon $\Pi^+$, we have constructed an initial
antiflip $\Pi^-$ with the property that $U_{\Pi^+}$ is symplectic
deformation equivalent to $U_{\Pi^-}$. This whole discussion was
inspired by results in Mori theory \cite{HTU}. Here is an alternative,
Mori-theoretic proof that $U_{\Pi^+}$ and $U_{\Pi^-}$ are
diffeomorphic:

\begin{Theorem}
Let $\Pi^+$ be a \(K\)-positive truncated wedge and let $\Pi^-$ be its
initial antiflip. The manifolds $U_{\Pi^+}$ and $U_{\Pi^-}$ are
diffeomorphic.
\end{Theorem}
\begin{proof}
The variety $V_{\Pi^-}$ admits a \QG smoothing
$\pi^-\colon\mathcal{V}^-\to\CC$. The curve $C_{\Pi^-}\subset
V_{\Pi^-}\subset\mathcal{V}^-$ is a
\(K_{\mathcal{V}_{\Pi^-}}\)-negative curve, and, in this situation,
Mori theory furnishes us with a {\em flip}
$\pi^+\colon\mathcal{V}^+\to\CC$ such that:
\begin{itemize}
\item $\pi^+\colon\mathcal{V}^+\to\CC$ is a \QG smoothing of
$V_{\Pi^+}$;
\item there is a biholomorphism $f\colon\mathcal{V}^+\setminus
C^+\to\mathcal{V}^-\setminus C^-$ such that $\pi^-=\pi^+\circ f$.
\end{itemize}
See \cite{HTU} (proof of Corollary 3.23, page 44 of arXiv version)
for a justification of the particular numbers involved in the
definitions of the polygons $\Pi^{\pm}$.

The smooth fibre of the \QG smoothing
$\pi^{\pm}\colon\mathcal{V}^{\pm}\to\CC$ is diffeomorphic to
$U_{\Pi^{\pm}}$, and since $\mathcal{V}^-$ and $\mathcal{V}^+$ are
fibre-preservingly biholomorphic away from the singular fibre this
means that $U_{\Pi^-}$ and $U_{\Pi^+}$ are diffeomorphic to one
another. \qedhere

\end{proof}
Of course, in Mori theory, a \QG smoothing with at worst canonical
singularities of {\em any} \(K\)-negative $V_\Pi$ (not necessarily an
initial antiflip) admits a flip. In terms of our pictures, the
algorithm to find the flip is to perform left mutations {\em down} the
Mori sequence until your \(K\)-negative polygon is not longer
left-mutable. At that point, one of two things happens:
\begin{itemize}
\item the polygon becomes left-immutable, in which case you perform the
flip as in Definition \ref{dfn:flip};
\item the polygon becomes borderline for left-mutability.
\end{itemize}
In the borderline case, $B_2$ is parallel to $R_1$. In this case,
there is a visible surface in the almost toric base $\Pi$, connecting
the singular point $z_2$ at the end of $B_2$ to the edge $R_1$
(visible surfaces are surfaces which project to paths in the almost
toric base; see (Definition 7.2, \cite{Symington})). This visible
surface is a symplectic \(-1\)-sphere (see Symington \cite{Symington},
Lemma 7.11). This corresponds to the phenomenon of {\em divisorial
contraction} in the minimal model programme; rather than the 3-fold
\QG smoothing of $V_\Pi$ admitting a flip along $C_\Pi$, a whole
surface can be contracted; this surface is the union of $C_\Pi$ and
all these visible \(-1\)-spheres.

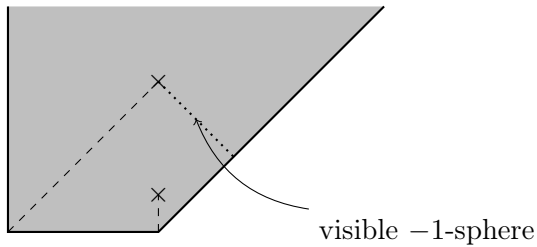
\begin{center}
\begin{tikzpicture}
\filldraw[draw=black,thick,fill=gray!50] (0,3) -- (0,0) -- (2,0) -- (5,3);
\draw[dashed] (0,0) -- (2,2) node {$\times$};
\draw[dashed] (2,0) -- (2,0.5) node {$\times$};
\draw[dotted,thick] (2,2) -- (3,1);
\node at (4,0) (label)[right] {visible \(-1\)-sphere};
\draw[->] (label.north west) to[bend left] (2.5,1.5);

\end{tikzpicture}
\end{center}
\begin{Remark}
We remark that the term ``antiflip'' is not always a well-defiend
operation in algebraic geometry: not only is there a whole Mori
sequence of antiflips, but it is entirely possible for a 3-fold
containing a curve $C$ with $K\cdot C>0$ (e.g. some \QG smoothings
of $V_\Pi$ for a \(K\)-positive $\Pi$) not to arise as a flip at
all. See \cite{HTU} for a discussion of when antiflips exist in the
algebro-geometric sense.

\end{Remark}
\subsection{Flips of type k1A}
\label{sct:k1A}

The paper \cite{HTU} also discusses flips where the \(K\)-negative
surface has only one Wahl singularity, obtained by \QG smoothing
$V_\Pi$ for some \(K\)-negative $\Pi$. We explain by example how this
situation arises in our almost toric pictures.

\begin{exm}
The following chain defines a \(K\)-negative polygon $\Pi$ such that
the QHP $U_\Pi$ is a symplectic filling of $L(11,3)$:
\[[2,5,3]-1-[2,3,2,2,7,3].\] If we \QG smooth the singularity
$[2,5,3]$ and take the minimal resolution of the other singularity
then we find a configuration of spheres $C_1,\ldots,C_6,E$, where
$\bigcup C_i$ is the exceptional locus of the minimal resolution
($[-C_1^2,\cdots,-C_6^2]=[2,3,2,2,7,3]$) and $E$ is a \(-1\)-sphere,
intersecting according to the following graph:

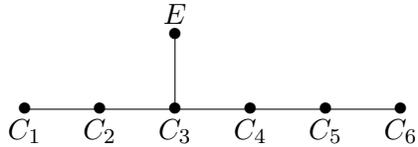
\begin{center}
\begin{tikzpicture}
\node (A) at (0,0) {$\bullet$};\node at (A) [below] {$C_1$};
\node (B) at (1,0) {$\bullet$};\node at (B) [below] {$C_2$};
\node (C) at (2,0) {$\bullet$};\node at (C) [below] {$C_3$};
\node (D) at (3,0) {$\bullet$};\node at (D) [below] {$C_4$};
\node (E) at (4,0) {$\bullet$};\node at (E) [below] {$C_5$};
\node (F) at (5,0) {$\bullet$};\node at (F) [below] {$C_6$};
\node (G) at (2,1) {$\bullet$};\node at (G) [above] {$E$};
\draw (A.center) -- (F.center);
\draw (C.center) -- (G.center);
\end{tikzpicture}
\end{center}

We can also understand this in terms of almost toric pictures. An
almost toric picture of the k1A neighbourhood can be obtained by
performing a single nodal trade the left-hand vertex of $\Pi$. The
minimal resolution of the other vertex can also be performed
torically. We now see the \(-1\)-sphere as a visible surface, since
the branch cut is parallel to the edge representing the sphere $C_3$
in the minimal resolution.

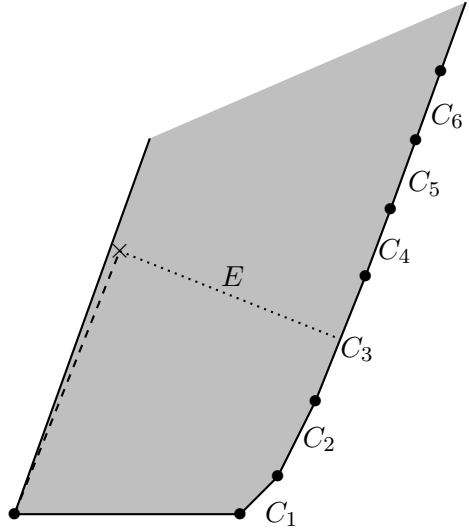
\begin{center}
\begin{tikzpicture}
\filldraw[fill=gray!50,thick,draw=black] (9/5,25/5) -- (0,0) node {$\bullet$} -- (3,0) node {$\bullet$} -- ++ (1/2,1/2) node {$\bullet$} -- ++ (1/2,2/2) node {$\bullet$}-- ++ (2/3,5/3) node {$\bullet$} -- ++ (3/9,8/9) node {$\bullet$} -- ++ (4/12,11/12) node {$\bullet$} -- ++ (25/75,69/75) node {$\bullet$} -- ++ (71/213,196/213);
\draw[thick,dashed] (0,0) -- (2*0.7,5*0.7) node {$\times$};
\draw[thick,dotted] (2*0.7,5*0.7) -- (4+96/290,3/2+48/58);
\node at (3.2,0) [right] {$C_1$};
\node at (3.7,1) [right] {$C_2$};
\node at (4.2,2.2) [right] {$C_3$};
\node at (4.7,3.5) [right] {$C_4$};
\node at (5.1,4.4) [right] {$C_5$};
\node at (5.4,5.3) [right] {$C_6$};
\node at (2*0.7+5*0.3,5*0.7-2*0.3) [above] {$E$};
\end{tikzpicture}
\end{center}

\end{exm}
In our picture, the k1A flip is no different from the k2A flip: one
simply performs one nodal trade and mutation at a time.

\subsection{A topological viewpoint}
\label{sct:top}

An almost toric structure on a truncated wedge $\Pi$ exhibits $U_\Pi$
as a handlebody obtained by attaching two Lagrangian 2-handles (the
pinwheel discs) to $S^1\times B^3$. The process of performing a flip
or initial antiflip is, topologically, a handleslide, from which point
of view it is clear that they are diffeomorphic.

On the other hand, if we think of them as smoothings of singular
orbifolds then the flip, initial antiflip and all the mutations can be
seen as compositions of well-known topological operations:
\begin{enumerate}
\item Find two rational homology balls $B_{p_i,q_i}$, $i=1,2$. Perform
generalised rational blow-up in both balls, yielding
Hirzebruch-Jung chains of exceptional spheres
$[b_{i,1},\ldots,b_{i,r_i}]$ representing the continued fractions
$\frac{p_i^2}{p_iq_i-1}$.
\item If you can find another curve $C$ in the rational blow-up with
self-intersection $-c$ such that the union of the Hirzebruch-Jung
chains and $C$ forms a chain
\[[b_{1,1},\ldots,b_{1,r_1}]-c-[b_{2,1},\ldots,b_{2,r_2}],\] then
continue.
\item Perform blow-up and blow-down on this chain to transform it into
another chain of the form
\[[b'_{1,1},\ldots,b'_{1,r'_1}]-c'-[b'_{2,1},\ldots,b'_{2,r'_2}].\]
with $[b'_{i,1},\ldots,b'_{i,r'_i}]=\frac{(p'_i)^2}{p'_iq'_i-1}$ for
some $p'_i,q'_i$, $i=1,2$.
\item Rationally blow down the bracketed Hirzebruch-Jung chains at either
end to obtain a new 4-manifold with two new rational homology balls
$B_{p'_i,q'_i}$, $i=1,2$.
\end{enumerate}
Such a string of operations need not yield a result diffeomorphic to
the manifold you started with; from this point of view, the fact that
the flip, initial antiflip and its mutations are all diffeomorphic is
something of a miracle.

\begin{exm}
Suppose we can rationally blow-up a $B_{2,1}$ to get a chain $[4]-3$
(we will see an example of this in the quintic surface
later). Starting with the chain $[4,3]$ we can blow up a point on
the \(-4\)-sphere (away from its intersection with the
\(-3\)-sphere) to get $[1,5,3]$, then once on the \(-1\)-sphere then
rationally blow-down the $[2,5,3]$ to get the initial antiflip. If
we want to get the first right mutation of the initial antiflip, we
continue blowing up and down:
\begin{gather*}
[1,2,5,3]\\
[1,5,3]\\
[2,1,6,3]\\
[2,2,1,7,3]\\
[2,3,1,2,7,3]\\
[2,4,1,2,2,7,3]\\
[2,5,1,2,2,2,7,3]\\
[2,5,2,1,3,2,2,7,3]\\
[2,5,3,1,2,3,2,2,7,3]
\end{gather*}
Finally, we can rationally blow-down $[2,5,3]$ and $[2,3,2,2,7,3]$
to get $B_{5,3}$ and $B_{14,9}$.

\end{exm}
\section{Examples}
\label{sct:examples}

Our examples will be built by smoothing certain singular surfaces to
find starting configurations of rational homology balls to which we
can apply Lemma \ref{lma:sufficientlysmall}.

\subsection{Symplectic smoothing}

We wish to consider three fillings of the lens space
\(L(p^2,pq-1)\):
\begin{enumerate}
\item the rational homology ball \(B_{p,q}\),
\item the singularity of type \(\frac{1}{p^2}(1,pq-1)\) (an orbifold filling),
\item the minimal resolution of this singularity.
\end{enumerate}
All three are almost toric (2 and 3 are actually toric) and have the
same contact boundary. Symington defines generalised rational blowdown
as the surgery of almost toric manifolds going from 3 to 1; in other
words, it is a surgery of almost toric symplectic manifolds defined by
performing surgery on the almost toric base diagrams. Similarly:

\begin{Definition}\label{dfn:sympsmooth}
We define symplectic smoothing as the surgery of almost toric orbifolds
going from 2 to 1.

\end{Definition}
\begin{exm}\label{exm:sympsmooth}
The symplectic smoothing of \(V_{\Pi}\) is \(U_\Pi\).

\end{exm}
We remark that you do not need a global almost toric structure to
perform these surgeries, only one over the region where the surgery is
taking place.

\begin{Lemma}\label{lma:sympalgsmooth}
If $V$ is a surface with Wahl singularities, which admits a \QG
smoothing whose total space supports a relatively ample line bundle,
then any smooth fibre of this smoothing is a surface
symplectomorphic to the symplectic smoothing $U$ of \(V\).
\end{Lemma}
\begin{proof}
The relatively ample line bundle yields a symplectic form on all the
fibres (away from the singular locus) and a symplectic connection on
this family of symplectic manifolds. The link of each Wahl
singularity in \(V\) is a lens space \(L(p^2,pq-1)\) of contact type
(equipped with a Milnor-fillable contact structure), and we can
symplectically parallel transport this link into the smooth
fibres. Each smooth fibre \(X\) therefore contains a separating
lens space \(\Sigma\) of contact type; we will write
\(X=U\cup_\Sigma (X\setminus U)\) where \(U\) is the region in \(X\)
which has \(\Sigma\) has convex (rather than concave) boundary. This
subset \(U\) is a symplectic filling of \(\Sigma\).

A \QG smoothing of a Wahl singularity has Milnor number zero, so
\(U\) is a rational homology ball. By Lisca's classification
\cite{Lisca} of symplectic fillings of lens spaces, \(U\) is
diffeomorphic to \(B_{p,q}\). Bhupal--Ono \cite{BhupalOno} showed
that this is a classification up to symplectic deformation, but
\(B_{p,q}\) has trivial second cohomology, so in this case it is a
classification up to symplectomorphism. Thus \(U\) is
symplectomorphic to \(B_{p,q}\), and \(X\) is obtained from \(V\) by
symplectic smoothing. \qedhere

\end{proof}
\subsection{Strategy}

We now explain how to construct examples of symplectically embedded
copies of $U_{\Pi^+}(\ell_1,\ell_2)$ in compact complex surfaces of
general type (for suitable \(K\)-positive polygons $\Pi^+$ and real
numbers $\ell_1,\ell_2$) using algebraic geometry. Then we will
perform the initial antiflip of the symplectic form and obtain an
infinitely mutable $U_{\Pi^-}(\ell_1,\ell_2)$ containing a Mori
sequence of Lagrangian pinwheels.

Recall that a KSBA-stable surface is a complex projective surface with
semi-log canonical singularities and ample dualising sheaf. If $V$ is
a KSBA-stable surface with at worst Wahl singularities then it is
\(\QQ\)-factorial, so we can replace this condition with having ample
canonical bundle; let $k$ be a positive integer such that
$K_V^{\otimes k}$ is very ample. Pulling back a Fubini-Study form
along the \(k\)-canonical embedding $V\to\mathbb{P}((H^0(K_V^{\otimes
k}\chk{)})$ and rescaling by $1/k$ furnishes $V$ with a K\"{a}hler
form $\omega$ satisfying $[\omega]=K_V$.

We can symplectically smooth the singularities of $V$ to obtain a
symplectic manifold $U$ as in Definition \ref{dfn:sympsmooth}. Suppose
that \(V\) is \QG smoothable. Since \(V\) is KSBA-stable, its
canonical bundle is ample, and since amplitude is an open condition,
the relative canonical bundle for this smoothing is ample (at least
for fibres near the singular fibre). By Lemma \ref{lma:sympalgsmooth},
the smooth fibre, which is necessarily a canonically polarised surface
of general type, is symplectomorphic to the symplectic smoothing
\(U\).

\begin{Theorem}\label{thm:ksba}
Let $V$ be a KSBA-stable surface with at worst Wahl
singularities. Suppose that $V$ contains a rational curve passing
through precisely two of its singularities $x_0$ and $x_1$ such that
$x_i$ is a Wahl singularity of type $\frac{1}{p_i^2}(1,p_iq_i-1)$,
and the preimage of $C$ in the minimal resolution of $X_0$ is a
chain \[[b_{0,1},\ldots,b_{0,r_0}]-c-[b_{1,1},\ldots,b_{1,r_1}],\]
with $\tilde{C}_0^2=-c$,
$\frac{p_i^2}{p_iq_i-1}=[b_{i,1},\ldots,b_{i,r_i}]$. Then the
symplectic smoothing $U$ contains a symplectically embedded copy of
$U_{\Pi^+}(\ell_0,\ell_1)$ for some $\ell_0,\ell_1>0$, where
$\Pi^+=\Pi(p_0,q_0,p_1,q_1,c,K_V\cdot C)$.

Let $\Pi^-=\Pi(p_1,q_1,p_2,q_2,1,a^-)$ be an initial right antiflip
of $\Pi^+$ with parameter $a^-$ sufficiently small and suppose that
$\Pi^-$ is infinitely right-mutable. Then the symplectic smoothing
$U$ admits a family of symplectic forms $\omega_t$ such that
$[\omega_0]=K$ and such that $\omega_1$ admits an infinite Mori
sequence $M(p_1,q_1;p_2,q_2)$ of Lagrangian pinwheels.
\end{Theorem}
\begin{proof}
By the symplectic neighbourhood theorem for symplectic suborbifolds
(Theorem 11, \cite{FinePanov}), a neighbourhood of $C$ in $V$ is
symplectomorphic to $V_{\Pi^+}(\ell_0,\ell_1)$ for some
$\ell_0,\ell_1>0$, where $\Pi^+=\Pi(p_0,q_0,p_1,q_1,c,a^+)$ and
$a^+$ is the symplectic area of $C$. Since $[\omega]=K_{V}$, this
means that $a^+=K_{V}\cdot C$.

The symplectic smoothing \(U\) of \(V\) is therefore obtained by
performing the symplectic smoothing on the almost toric region
\(V_{\Pi^+}(\ell_0,\ell_1)\), which (as in Example
\ref{exm:sympsmooth}) yields a copy of \(U_{\Pi^+}(\ell_0,\ell_1)\)
inside \(U\).

The initial right antiflip $U'$ of $U$ along $U_{\Pi^+}$ with
parameter $a^-$ is a symplectically embedded copy of
$U_{\Pi^-}(\ell'_1,\ell'_2)$ for some $\ell'_1,\ell'_2$, where
$\Pi^-$ is the initial right antiflip polygon of $\Pi^+$ with
parameter $a^-$. By Lemma \ref{lma:sufficientlysmall}, if $a^-$ is
sufficiently small then $U'$ admits the required Mori sequence of
Lagrangian pinwheels. \qedhere

\end{proof}
\subsection{The quintic surface}

\begin{Lemma}\label{lma:quintic}
There exists a KSBA-stable surface $V$ with $K^2=5$, $p_g=4$ with a
single singularity of type $\frac{1}{4}(1,1)$ such that its minimal
resolution contains a chain of spheres: \[[4]-3.\] Moreover, $V$
admits a \QG smoothing whose smooth fibre is a quintic surface.
\end{Lemma}
\begin{proof}
Following Rana \cite{Rana}, observe that the minimal resolution of a
stable quintic surface with a $\frac{1}{4}(1,1)$ singularity is a
Horikawa surface with $K^2=p_g=4$ containing a
\(-4\)-sphere. Moreover, such stable quintic surfaces $V$ are always
\QG smoothable, since the local-to-global obstruction group
$H^2(V,T_V)$ vanishes by (\cite{Rana}, Theorem 4.10). Let
$B\subset\cp{1}\times\cp{1}$ be a curve of bidegree $(6,6)$; the
branched double cover of $\cp{1}\times\cp{1}$ branched over $B$ is a
Horikawa surface of the required type.
\begin{itemize}
\item If $B$ intersects the diagonal at six points each with
multiplicity 2 then the preimage of the diagonal contains two
irreducible rational \(-4\)-spheres (intersecting at four points).
\item If $B$ intersects $\cp{1}\times\{z\}$ at three points each with
multiplicity 2 then the preimage of this ruling is a pair of
rational \(-3\)-spheres (intersecting at three points).
\end{itemize}
If we have found such a $B$ then we obtain a $[4]-3$ configuration
in the minimal resolution of a stable quintic.

One can verify that the curve $B$ given in the affine chart
$([x:1],[y:1])$) by $\{1-2y^3+y^6+2x^3-xy^5-2x^5y+x^6y^6=0\}$ has the
required properties: it is smooth, it intersects the ruling
$\{x=0\}$ at the three points $(0,\mu)$, $\mu^3=1$, each with
multiplicity two, and it intersects the diagonal at the six points
$(\mu,\mu)$, $\mu^6=1$, each with multiplicity two. \qedhere

\end{proof}
By Theorem \ref{thm:ksba}, this implies that the smooth quintic
surface contains a symplectically embedded $U_{\Pi^+}(\ell_1,\ell_2)$
where $\Pi^+=\Pi(2,1,1,1,3,a^+)$ with $a^+=K_V\cdot
C=\frac{\delta}{p_1p_2}=\frac{3}{2}$, and that its initial right
antiflip contains a Mori sequence of Lagrangian pinwheels. In this
case, we have $\delta=3$ and the initial antiflip polygon is
$\Pi^-=\Pi(1,0,5,3,1,a^-)$, so the relevant Mori sequence is
$M(1,0;5,3)$.

\subsection{A Godeaux surface}

\begin{Lemma}\label{lma:godeaux}
There exists a KSBA-stable surface $V$ with $K^2=1$, $p_g=0$ with
one ordinary double point and four Wahl singularities with continued
fractions \[[7,2,2,2], [3,5,2], [6,2,2], [4],\] such that its
minimal resolution contains a chain of spheres:
\[[2,2,6]-1-[3,5,2]\] Moreover, $V$ admits a \QG smoothing whose
smooth fibre is a simply-connected Godeaux surface.
\end{Lemma}
\begin{proof}
This surface is constructed in (\cite{Urzua}, Section 5) by flipping
an example of Lee and Park \cite{LP07}. \qedhere

\end{proof}
Below, we reproduce Figure 5 from \cite{Urzua} which illustrates a
configuration of curves in the minimal resolution of $V$ including the
chain we want (in red). The solid curves are collapsed by the minimal
resolution to give the ordinary double point and four Wahl
singularities of $V$. The dashed curves become rational curves in $V$.

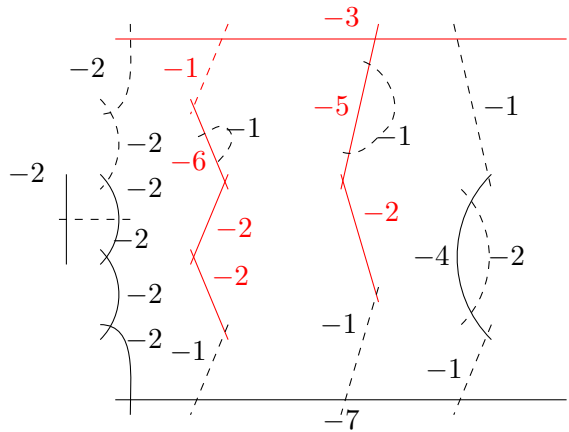
\begin{center}
\begin{tikzpicture}
\draw[red] (0,0) -- (6,0) node [above,midway] {$-3$};
\draw (0,-4.8) -- (6,-4.8) node [below,midway] {$-7$};
\draw[dashed] (0.2,0.2) to[out=-90,in=0] (-0.2,-1);
\draw[dashed] (-0.2,-0.8) to[out=-45,in=45] (-0.2,-2);
\draw (-0.2,-1.8) to[out=-45,in=45] (-0.2,-3);
\draw (-0.2,-2.8) to[out=-45,in=45] (-0.2,-4);
\draw (0.2,-5) to[out=90,in=0] (-0.2,-3.8);
\draw[dashed] (0.2,-2.4) -- (-0.8,-2.4);
\draw (-0.65,-1.8) -- (-0.65,-3);
\node at (0,-0.4) [left] {$-2$};
\node at (0,-1.4) [right] {$-2$};
\node at (0,-2) [right] {$-2$};
\node at (0,-3.4) [right] {$-2$};
\node at (0,-4) [right] {$-2$};
\node at (0.2,-2.4) [below] {$-2$};
\node at (-0.8,-1.8) [left] {$-2$};
\draw[red,dashed] (1.5,0.2) -- (1,-1) node [midway,left] {$-1$};
\draw[red] (1,-0.8) -- (1.5,-2) node [pos=0.7,left] {$-6$};
\draw[red] (1.5,-1.8) -- (1,-3) node [pos=0.6,right] {$-2$};
\draw[red] (1,-2.8) -- (1.5,-4) node [pos=0.3,right] {$-2$};
\draw[dashed] (1.5,-3.8) -- (1,-5) node [pos=0.3,left] {$-1$};
\draw[dashed] (1.1,-1.3) to[out=20,in=135] (1.5,-1.2) to[out=-45,in=45] (1.3,-1.7);
\node at (1.7,-1.2) {$-1$};
\draw[red] (3.5,0.2) -- (3,-2) node [left,midway] {$-5$};
\draw[red] (3,-1.8) -- (3.5,-3.5) node [right,pos=0.3] {$-2$};
\draw[dashed] (3.5,-3.3) -- (3,-5) node [left,pos=0.3] {$-1$};
\draw[dashed] (3.3,-0.3) to[out=-20,in=45] (3.6,-1.2) to[out=-135,in=-20] (3,-1.5);
\node at (3.7,-1.3) {$-1$};
\draw[dashed] (4.5,0.2) -- (5,-2) node [right,midway] {$-1$};
\draw (5,-1.8) to[out=-135,in=135] (5,-4);
\node at (4.2,-2.9) {$-4$};
\draw[dashed] (5,-3.8) -- (4.5,-5) node [left,midway] {$-1$};
\draw[dashed] (4.6,-2) to[out=-45,in=45] (4.6,-3.8);
\node at (5.2,-2.9) {$-2$};

\end{tikzpicture}
\end{center}
Theorem \ref{thm:ksba} implies that the simply-connected Godeaux
surface obtained by smoothing $V$ contains a symplectically embedded
$U_{\Pi^+}(\ell_1,\ell_2)$ where $\Pi^+=\Pi(4,3,5,2,1,a^+)$ with
$a^+=K_V\cdot C=\frac{\delta}{4\times 5}=\frac{7}{20}$, and that its
initial right antiflip contains a Mori sequence of Lagrangian
pinwheels. In this case, we have $\delta=7$ and the initial antiflip
polygon is $\Pi^-=\Pi(5,2,39,17,1,a^-)$, so the relevant Mori sequence
is $M(5,2;39,17)$.

\bibliographystyle{plain}
\bibliography{antiflips}
\end{document}